\documentclass{article}
\usepackage{amsmath,amssymb,amsthm}
\usepackage{graphicx,subfigure,float,url,color}
\usepackage{mathrsfs}
\usepackage[colorlinks=true]{hyperref}
\usepackage{pdfsync}

\def\R{\textrm{I\kern-0.21emR}}
\def\N{\textrm{I\kern-0.21emN}}
\def\Z{\mathbb{Z}}
\newcommand{\C}{\mathbb{C}}

\renewcommand{\geq}{\geqslant}
\renewcommand{\leq}{\leqslant}

\newtheorem{theorem}{Theorem}  
\newtheorem{proposition}{Proposition}

\newtheorem{lemma}{Lemma}

\theoremstyle{definition}\newtheorem{remark}{Remark}

\newcommand{\Op}{\mathrm{Op}}
\newcommand{\QL}{\mathcal{Q}}
\newcommand{\I} {\mathcal{I}}
\newcommand{\supp}{\mathrm{supp}}

\title{Geometric and spectral characterization of Zoll manifolds, invariant measures and quantum limits}

\author{Emmanuel Humbert\footnote{Laboratoire de Math\'ematiques et de Physique Th\'eorique, UFR Sciences et Technologie, Facult\'e Fran\c cois Rabelais, Parc de Grandmont, 37200 Tours, France (\texttt{emmanuel.humbert@lmpt.univ-tours.fr}).}
\and Yannick Privat\footnote{IRMA, Universit\'e de Strasbourg, CNRS UMR 7501, 7 rue Ren\'e Descartes, 67084 Strasbourg, France ({\tt yannick.privat@unistra.fr}).}
	\and Emmanuel Tr\'elat\footnote{Sorbonne Universit\'e, Universit\'e Paris-Diderot SPC, CNRS, Inria, Laboratoire Jacques-Louis Lions, \'equipe CAGE, F-75005 Paris (\texttt{emmanuel.trelat@sorbonne-universite.fr}).}}

\date{}

\begin{document}

\maketitle

\begin{abstract}
We provide new geometric and spectral characterizations for a Riemannian manifold to be a Zoll manifold, i.e., all geodesics of which are periodic. We analyze relationships with invariant measures and quantum limits.
\end{abstract}

\section{Introduction and main results}

Let $(M,g)$ be a closed connected smooth Riemannian manifold of finite dimension $n\in\N^*$, endowed with its canonical Riemannian measure $dx_g$.  

We denote by $T^*M$ the cotangent bundle of $M$ and by $S^*M$ the unit cotangent bundle, endowed with the Liouville measure $\mu_L$, and we denote by $\omega=-d\mu_L$ the canonical symplectic form on $T^*M$.
We consider the Riemannian geodesic flow $(\varphi_t)_{t\in\R}$, where, for every $t\in\R$, $\varphi_t$ is a symplectomorphism on $(T^*M,\omega)$ which preserves $S^*M$. 
A geodesic is a curve $t\mapsto\varphi_t(z)$ on $S^*M$ for some $z\in S^*M$.
Throughout the paper, we denote by $\pi:T^*M\rightarrow M$ the canonical projection. The same notation is used to designate the restriction of $\pi$ to $S^*M$. A (geodesic) ray $\gamma$ is a curve on $M$ that is the projection onto $M$ of a geodesic curve on $S^*M$, that is, $\gamma(t) = \pi\circ\varphi_t(z)$ for some $z\in S^*M$. We denote by $\Gamma$ the set of all geodesic rays.

Given any $k\in\N^*$, we denote by $\mathcal{S}^k(M)$ the set of classical symbols of order $k$ on $M$, and by $\Psi^k(M)$ the set of pseudodifferential operators of order $k$ (see \cite{Hormander,Zworski}). Choosing a quantization $\Op$ on $M$ (for instance, the Weyl quantization), given any $a\in \mathcal{S}^k(M)$, we have $\Op(a)\in \Psi^k(M)$. Any element of $\Psi^0(M)$ is a bounded endomorphism of $L^2(M,dx_g)$.

Throughout the paper, we denote by $\langle\ ,\ \rangle$ the scalar product in $L^2(M,dx_g)$ and by $\Vert\ \Vert$ the corresponding norm.

We consider the Laplace-Beltrami operator $\triangle$ on $(M,g)$. Its positive square root, $\sqrt{\triangle}$, is a selfadjoint pseudodifferential operator of order one, of principal symbol $\sigma_P(\sqrt{\triangle})=g^\star$, the cometric of $g$ (defined on $T^*M$). The spectrum of $\sqrt{\triangle}$ is discrete and is denoted by $\mathrm{Spec}(\sqrt{\triangle})$. The set of normalized (i.e., of norm one in $L^2(M,dx_g)$) real-valued eigenfunctions $\phi$ is denoted by $\mathcal{E}$.

\medskip

We say that the manifold $M$ is Zoll whenever all its geodesics are periodic (see \cite{Besse}). Zoll manifolds have been characterized within a spectral viewpoint in \cite{Guillemin_Duke,Helton}, where it has been shown that, in some sense, periodicity of geodesics is equivalent to periodicity in the spectrum of $\sqrt{\triangle}$ (see Section \ref{sec:knownresults} for details).

Given any $T>0$ and any Lebesgue measurable subset $\omega$ of $M$, denoting by $\chi_\omega$  the characteristic function of $\omega$, we define the geometric quantities
$$
g_2^T(\omega) = \inf_{\gamma\in\Gamma} \frac{1}{T} \chi_\omega(\gamma(t))\, dt \quad\textrm{and}\quad
g_2(\omega) = \liminf_{T\rightarrow +\infty} g_2^T(\omega)  
$$
and, denoting by $\mathcal{E}$ the set of eigenfunctions $\phi$ of $\sqrt{\triangle}$ of norm one in $L^2(M,dx_g)$, we define the spectral quantities
$$
g_1(\omega) = \inf_{\phi\in\mathcal{E}} \int_\omega \phi^2\, dx_g
$$
and, for $\omega$ Borel measurable, 
$$
g_1''(\omega) = \inf_{\mu\in \I(S^*M)} \mu(\pi^{-1}(\omega))
$$
where $\I(S^*M)$ is the set of probability Radon measures $\mu$ on $S^*M$ that are invariant under the geodesic flow. It is always true that $g_2^T(\omega)\leq g_2(\omega)\leq g_1''(\omega)$.
These geometric and spectral functionals are defined in a more general setting in Sections \ref{sec:geometricquantities} and \ref{sec:spectralquantities}, as well as several others which are of interest. 

Our main result (Theorem \ref{mainthm}), formulated in Section \ref{sec:mainresults}, provides new characterizations of Zoll manifolds and relations with quantum limits, among which we quote the following:

\medskip
\boxed{
\begin{minipage}{13.5cm}
\begin{itemize}
\item $M$ is Zoll $\ \Longleftrightarrow\ $ $g_2(\omega)=g_1''(\omega)$ for every $\omega\subset M$ Borel measurable.
\item $M$ is Zoll and the Dirac measure $\delta_\gamma$ along any periodic ray $\gamma\in\Gamma$ is a quantum limit on $M$\\
$\ \Longleftrightarrow\ $ there exists $T>0$ such that $g_1(\omega) \leq g_2^T(\omega)$ for any closed subset $\omega\subset M$.
\item If $g_1(\omega) \leq g_1''(\omega)$ for any closed subset $\omega\subset M$ then the Dirac measure $\delta_\gamma$ along any periodic ray $\gamma\in\Gamma$ is a quantum limit on $M$.
\item Assume that the spectrum of $\sqrt{\triangle}$ is uniformly locally finite\footnote{This means that there exists $\ell>0$ and $m\in\N^*$ such that the intersection of the spectrum with any interval of length $\ell$ has at most $m$ {\em distinct} elements (allowing multiplicity with arbitrarily large order).}. Then 
 $M$ is Zoll and for every geodesic ray $\gamma$ there exists a quantum limit $\mu$ on $M$ such that $\mu(\gamma(\mathbb{R})) >0$.
\end{itemize}
\end{minipage}
}

\medskip 

\noindent Here, a quantum limit on $M$ is defined as a probability Radon measure on $M$ that is a weak limit of a sequence of probability measures $\phi_\lambda^2\, dx_g$. 
The last item above slightly generalizes some results of \cite{Guillemin_Duke, Helton, Macia_CPDE2008}.

\medskip

The study of $g_1(\omega)$ and $g_2(\omega)$ done in this paper is in particular motivated by the following inequality on the observability constant for the wave equation on $M$:
$$
\frac{1}{2}\min(g_1(\omega), g_2(\mathring{\omega}))  \leq \lim_{T \to +\infty} \frac{C_T(\omega)}{T}  \leq \frac{1}{2} \min(g_1(\omega),g_2(\overline\omega)) 
$$ 
which is valid for any Lebesgue measurable subset $\omega$ of $M$ (see \cite[Theorems 2 and 3]{HPT1}). Here, given any $T>0$, the observability constant $C_T(\omega)$ is defined as the largest possible nonnegative constant $C$ such that the observability inequality
\begin{equation*} 
C \Vert (y(0,\cdot),\partial_ty(0,\cdot))\Vert_{L^2(M)\times H^{-1}(M)}^2 \leq \int_0^T\int_\omega \vert y(t,x)\vert^2 \,dx_g \, dt
\end{equation*}
is satisfied for all possible solutions of the wave equation $\partial_{tt}y- \triangle y=0$.

\medskip

The article is organized as follows. In Sections \ref{sec:geometricquantities} and \ref{sec:spectralquantities}, we define with full details the various geometric and spectral quantities that are of interest for the forthcoming results. In Section \ref{sec:knownresults}, we recall several known results about new characterizations of Zoll manifolds. In Section \ref{sec:mainresults}, we gather all the new results and estimates about the characterization of Zoll manifolds and the relations with quantum limits. Finally, the proofs of these results are all postponed to Section \ref{sec:proofs}.

\subsection{Geometric quantities}\label{sec:geometricquantities}
Given any bounded measurable function $a$ on $(S^*M,\mu_L)$ and given any $T>0$, we define
$$
\boxed{
g_2^T(a) = \inf_{z\in S^*M} \frac{1}{T} \int_0^T a\circ\varphi_t(z)\, dt =  \inf_{z\in S^*M} \bar a_T(z)
}
$$
where $\bar a_T(z) = \frac{1}{T} \int_0^T a\circ\varphi_t(z)\, dt$. Note that $g_2^T(a)=g_2^T(a\circ\varphi_t)$, i.e., $g_2^T$ is invariant under the geodesic flow. We set
$$
\boxed{
g_2(a) = \liminf_{T\rightarrow +\infty} \underset{g_2^T(a)}{\underbrace{\inf_{z\in S^*M} \frac{1}{T} \int_0^T a\circ\varphi_t(z)\, dt}} = \liminf_{T\rightarrow +\infty} \inf_{z\in S^*M}\bar a_T(z)
}
$$
and 
$$
\boxed{
g_2'(a) =  \inf_{z\in S^*M}  \liminf_{T\rightarrow +\infty} \frac{1}{T} \int_0^T a\circ\varphi_t(z)\, dt = \inf_{z\in S^*M} \liminf_{T\rightarrow +\infty} \bar a_T(z)
}
$$
Note that we always have $g_2(a)\leq g_2'(a)$.

Given two functions on $S^*M$ such that $a=\tilde a$ $\mu_L$-almost everywhere, we may have $g_2(a)\neq g_2(\tilde a)$ and $g_2'(a)\neq g_2'(\tilde a)$. Indeed, taking $a=1$ everywhere on $S^*M$ and $\tilde a=1$ as well except along one geodesic, one has $g_2(a)=1$ and $g_2(\tilde a)=0$, although $a=\tilde a$ almost everywhere.

Note that $g_2^T$, $g_2$ and $g_2'$ are inner measures on $S^*M$, which are invariant under the geodesic flow. They are superadditive but not subadditive in general (and thus, they are not measures).

We can pushforward them to $M$ under the canonical projection $\pi:S^*M\rightarrow M$: given any bounded measurable function $f$ on $(M,dx_g)$, we set
$$
(\pi_*g_2^T)(f) = g_2^T(\pi^*f) = g_2^T(f\circ\pi) = \inf_{z\in S^*M} \frac{1}{T} \int_0^T f\circ\pi\circ\varphi_t(z)\, dt = \inf_{\gamma\in\Gamma} \frac{1}{T} \int_0^T f(\gamma(t))\, dt 
$$
and accordingly,
$$
(\pi_*g_2)(f) = \liminf_{T\rightarrow +\infty} \underset{g_2^T(f)}{\underbrace{\inf_{\gamma\in\Gamma} \frac{1}{T} \int_0^T f(\gamma(t))\, dt}}, \qquad
(\pi_*g_2')(f) =  \inf_{\gamma\in\Gamma}  \liminf_{T\rightarrow +\infty} \frac{1}{T} \int_0^T f(\gamma(t))\, dt ,
$$
that we simply denote by $g_2^T(f)$, $g_2(f)$ and $g_2'(f)$ respectively when the context is clear.
Also, given any Lebesgue measurable\footnote{Here, measurability is considered in the Lebesgue sense, that is, for instance, for the measure $\pi_*\mu_L$ on $M$.} subset $\omega$ of $M$, denoting by $\chi_\omega$ the characteristic function of $\omega$, defined by $\chi_\omega(x)=1$ if $x\in\omega$ and $\chi_\omega(x)=0$ otherwise, we often denote by $g_2^T(\omega)$, $g_2(\omega)$ and $g_2'(\omega)$ instead of $g_2^T(\chi_\omega)$, $g_2(\chi_\omega)$ and $g_2'(\chi_\omega)$ respectively.
Note that the real number
$$\liminf_{T\rightarrow +\infty} \frac{1}{T} \int_0^T \chi_{\omega}(\gamma(t))\, dt$$
represents the average time spent by the ray $\gamma$ in $\omega$.

It is interesting to notice that, for $\omega\subset M$ open, if $g_2(\omega)=0$ then $g_2'(\omega)=0$ (see Lemma \ref{lemg2g2prime}).

\begin{remark}\label{rem_g_2_sup}
Given any bounded measurable function $a$ on $(S^*M,\mu_L)$, we have 
$$
g_2^T(a) \leq g_2(a)\quad\forall T>0\qquad\textrm{and}\qquad
g_2(a) = \lim_{T\rightarrow+\infty}g_2^T(a)=\sup_{T>0} g_2^T(a) = \sup_{T>0} \inf \bar a_T .
$$
Indeed, let $T_m$ converging to $+\infty$ such that $\lim_m g_2^{T_m}(a) = g_2(a)$. In the following, $\lfloor x\rfloor$ denotes the integer part of the real number $x$. Writing $T_m = \lfloor \frac{T_m}{T} \rfloor T + \delta_m$ for some $\delta_m \in [0,T]$, and setting $n_m=  \lfloor \frac{T_m}{T} \rfloor$, we have
$$
g_2^{T_m}(a) = \inf \left( \frac{1}{ T_m} \int_0^{n_mT}  a\circ\varphi_t\, dt + \frac{1}{T_m} \int_{n_mT}^{n_M T+ \delta_m}   a\circ\varphi_t\, dt \right) 
\geq  \inf \left( \frac{1}{ T_m} \sum_{k=0}^{n_m-1} \int_{kT}^{(k+1)T}  a\circ\varphi_t\, dt \right) . 
$$
Noting that $\frac{1}{T}  \int_{kT}^{(k+1)T}  a\circ\varphi_t\, dt \geq g_2^T(a)$ for every $k$, we obtain $g_2^{T_m}(a) \geq \frac{n_m T}{T_m}g_2^T(a)$. 
The claim follows by letting $T_m$ tend to $+\infty$. Note that this argument is exactly the one used to establish Fekete's Lemma: indeed, for $a$ fixed the function $T\mapsto T g_2^T(a)$ is superadditive.
\end{remark}

\begin{remark}
We have $g_2^T(a)\leq \mu(a)$ for every $T>0$, for every Borel measurable function $a$ on $S^*M$, and for every probability measure $\mu$ on $S^*M$ that is invariant under the geodesic flow.
We will actually establish in Lemma \ref{leminegalites} a more general result.
%
\end{remark}

\begin{remark}\label{remEgorov}
Setting $a_t=a\circ\varphi_t$, and assuming that $a\in C^0(S^*M)$ is the principal symbol of a pseudo-differential operator $A\in\Psi^0(M)$ (of order $0$), that is, $a=\sigma_P(A)$, we have, by the Egorov theorem (see \cite{Hormander,Zworski}),
$$
a_t = a\circ\varphi_t = \sigma_P(A_t)\qquad\textrm{with}\qquad A_t = e^{-it\sqrt{\triangle}}Ae^{it\sqrt{\triangle}}
$$
where $\sigma_P(\cdot)$ is the principal symbol.
Accordingly, we have $\bar a_T=\sigma_P(\bar A_T)$ with 
$$
\bar A_T=\frac{1}{T}\int_0^T A_t\, dt=\frac{1}{T}\int_0^T e^{-it\sqrt{\triangle}}Ae^{it\sqrt{\triangle}}\, dt .
$$
We provide hereafter a microlocal interpretation of the functionals $g_2^T$, $g_2$, $g_2'$ and give a relationship with the wave observability constant.
\end{remark}

\paragraph{Microlocal interpretation of $g_2^T$, $g_2$, $g_2'$, and of the wave observability constant.}
Let $f_T$ be such that $\hat f_T(t) = \frac{1}{T}\chi_{[0,T]}(t)$, i.e., $f_T(t)=\frac{1}{2\pi}ie^{iTt/2}\mathrm{sinc}(Tt/2)$).
Note that $\int_\R \hat f_T=1$, i.e., equivalently, $f_T(0)=1$.
Using that $a\circ e^{tX} = (e^{tX})^*a=e^{tL_X}a=e^{itS}a$, we get
$$
g_2^T(a) = \inf_{z\in S^*M} \frac{1}{T} \int_0^T a\circ e^{tX}(z)\, dt
= \inf_{z\in S^*M} \int_\R \hat f_T(t) e^{itS}a\, dt\ (z)
= \inf f_T(S)a .
$$
Besides, setting $A=\Op(a)$, we have
$$
\bar A_T(a) = \frac{1}{T} \int_0^T e^{-it\sqrt{\triangle}} a e^{it\sqrt{\triangle}}\, dt
= \int_\R \hat f_T(t) e^{-it\sqrt{\triangle}} a e^{it\sqrt{\triangle}}\, dt
= A_{f_T} = \sum_{\lambda,\mu} f_T(\lambda-\mu) P_\lambda A P_\mu.
$$
Restricting to half-waves, the wave observability constant is therefore given (see \cite{HPT1}) by
$$
C_T(a) = \inf_{\Vert y\Vert=1} \langle \bar A_T(a) y,y\rangle = \inf_{\Vert y\Vert=1} \langle A_{f_T} y,y\rangle .
$$
Note that
\begin{multline*}
\langle A_{f_T} y,y\rangle = \sum_{\lambda,\mu} f_T(\lambda-\mu) \langle AP_\lambda y, P_\mu y\rangle
= \sum_{\lambda,\mu} f_T(\lambda-\mu) \int_M a \, P_\lambda y\, \overline{P_\mu y} \\
= \sum_{\lambda,\mu} f_T(\lambda-\mu) a_\lambda \bar a_\mu \int_M a \phi_\lambda \phi_\mu 
\end{multline*}
and we thus recover the expression of $C_T(a)$ by series expansion.

Note also that, as said before, the principal symbol of $A_{f_T}=\bar A_T(a)$ is 
$$
\sigma_P(A_{f_T}) = \sigma_P(\bar A_T(a)) = a_{f_T} = \int_\R \hat f_T(t) a\circ e^{tX}\, dt = f_T(S)a
$$
and that $g_2^T(a) = \inf \sigma_P(\bar A_T(a))$.

Note as well that
$$
g_2'(a) = \inf_{S^*M} \liminf_{T\rightarrow+\infty} f_T(S) a
$$
and that $f_T$ converges pointwise to $\chi_{\{0\}}$ as $T\rightarrow +\infty$, and uniformly to $0$ outside of $0$. Since $S=\frac{1}{i}L_X$ is selfadjoint with compact inverse, it has a discrete spectrum $0=\mu_0<\mu_1<\cdots$ associated with eigenfunctions $\psi_j$. If $a=\sum a_j\psi_j$, then $f_T(S)a=\sum f_T(\mu_j)a_j\psi_j\rightarrow a_0\psi_0$ as $T\rightarrow +\infty$. In other words, we have
$$
g_2'(a) = \inf Q_0 a
$$
where $Q_0$ is the projection onto the eigenspace of $S$ associated with the eigenvalue $0$, which is also the set of functions that are invariant under the geodesic flow.

\subsection{Spectral quantities}\label{sec:spectralquantities}
Recall that $\mathcal{E}$ is the set of normalized (i.e., of norm one in $L^2(M,dx_g)$) real-valued eigenfunctions $\phi$ of $\sqrt{\triangle}$.
Choosing a quantization\footnote{A quantization is constructed by covering the closed manifold $M$ with a finite number of coordinate charts; once this covering is fixed, by using a smooth partition of unity, we define the quantization of symbols that are supported in some coordinate charts, and there we choose a quantization, for instance the Weyl quantization.}
$\Op$ on $M$, given any symbol $a\in\mathcal{S}^0(M)$ of order $0$, we define 
$$
\boxed{
g_1(a) = \inf_{\phi\in\mathcal{E}} \langle\Op(a)\phi,\phi\rangle
}
$$
Note that this definition depends on the chosen quantization. In order to get rid of the quantization, one could define $g_1(A) = \inf_{\phi\in\mathcal{E}} \langle A\phi,\phi\rangle$ for every $A\in\Psi^0(M)$ that is nonnegative and selfadjoint.
Note that $g_1(A)$ is then the infimum of eigenvalues of the  operator $A$ on $L^2(M,dx_g)$. 

As we have done for $g_2$, we pushforward the functional $g_1$ to $M$ under the canonical projection $\pi$, and we set (noting that $\Op(f\circ\pi)\phi=f\phi$)
$$
(\pi_*g_1)(f) = g_1(f\circ\pi) = \inf_{\phi\in\mathcal{E}} \int_M f \phi^2\, dx_g
$$
for every $f\in C^0(M)$, which we also denote by $g_1(f)$. Note that the definition of $g_1(f)$ still makes sense for essentially bounded Lebesgue measurable functions $f$ on $(M,dx_g)$ which need not be continuous.

Accordingly, given any Lebesgue measurable subset $\omega$ of $M$, we will often denote by 
$$
\boxed{
g_1(\omega) = \inf_{\phi\in\mathcal{E}} \int_\omega \phi^2\, dx_g
}
$$
to designate the quantity $g_1(\chi_\omega)$.
Note that, like $g_2$, the functional $g_1$ is an inner measure (which is not sub-additive in general).

\begin{remark}\label{remg1}
It is interesting to note that, given any Lebesgue measurable subset $\omega$ of $M$ such that $\partial\omega=\overline{\omega}\setminus\mathring{\omega}$ has zero Lebesgue measure (i.e., $\omega$ is a Jordan measurable set), we have
$$
g_1(\mathring{\omega}) = g_1(\omega) = g_1(\overline{\omega}).
$$
Indeed, in this case we have $\int_{\mathring{\omega}} \phi^2\, dx_g = \int_\omega \phi^2\, dx_g = \int_{\overline{\omega}} \phi^2\, dx_g$ for every $\phi\in\mathcal{E}$.

More generally, we have $g_1(f)=g_1(\tilde f)$ for all essentially bounded Lebesgue measurable functions $f$ and $\tilde f$ coinciding Lebesgue almost everywhere on $(M,dx_g)$.
\end{remark}


\paragraph{Quantum limits.}
We recall that a \textit{quantum limit} (QL in short) $\mu$, also called \textit{semi-classical measure}, is a probability Radon (i.e., probability Borel regular) measure on $S^*M$ that is a closure point (weak limit), as $\lambda\rightarrow+\infty$, of the family of Radon measures $\mu_\lambda(a)=\langle\Op(a)\phi_\lambda,\phi_\lambda\rangle$ (which are asymptotically positive by the G\aa rding inequality), where $\phi_\lambda$ denotes an eigenfunction of norm $1$ associated with the eigenvalue $\lambda$ of $\sqrt{\triangle}$. We speak of a \textit{QL on $M$} to refer to a closure point (for the weak topology) of the sequence of probability Radon measures $\phi_\lambda^2\, dx_g$ on $M$ as $\lambda\rightarrow +\infty$. 
Note that QLs do not depend on the choice of a quantization. We denote by $\QL(S^*M)$ (resp., $\QL(M)$) the set of QLs (resp., the set of QLs on $M$). Both are compact sets.

Given any $\mu\in\QL(S^*M)$, the Radon measure $\pi_*\mu$, image of $\mu$ under the canonical projection $\pi:S^*M\rightarrow M$, is a probability Radon measure on $M$. It is defined, equivalently, by $(\pi_*\mu)(f) = \mu(\pi^*f) = \mu(f\circ\pi)$ for every $f\in C^0(M)$ (note that, in local coordinates $(x,\xi)$ in $S^*M$, the function $f\circ\pi$ is a function depending only on $x$), or by $(\pi_*\mu)(\omega)=\mu(\pi^{-1}(\omega))$ for every $\omega\subset M$ Borel measurable (or Lebesgue measurable, by regularity).
It is easy to see that\footnote{Indeed, given any $f\in C^0(M)$ and any $\lambda\in\mathrm{Spec}(\sqrt{\triangle})$, we have
$$
(\pi_*\mu_\lambda)(f) = \mu_\lambda(\pi^*f) = \langle\Op(\pi^*f)\phi_\lambda,\phi_\lambda\rangle = \int_M f \phi_\lambda^2\, dx_g ,
$$
because $\Op(\pi^*f)\phi_\lambda=f\phi_\lambda$. The equality then easily follows by weak compactness of probability Radon measures.
}
\begin{equation}\label{eq_QLbase}
\pi_* \QL(S^*M) = \QL(M) .
\end{equation}
In other words, QLs on $M$ are exactly the image measures under $\pi$ of QLs.


\medskip

Given any 
bounded Borel measurable function $a$ on $S^*M$, we define
$$
\boxed{
g_1'(a) = \inf_{\mu\in\QL(S^*M)} \int_{S^*M}a\, d\mu 
}
$$
As before, we pushforward the functional $g_1'$ to $M$, by setting $(\pi_*g_1')(f) = g_1'(f\circ\pi)$ for every $f\in C^0(M)$, which we often denote by $g_1'(f)$. Thanks to \eqref{eq_QLbase}, we have
$$
g_1'(f) = \inf_{\nu\in\QL(M)} \nu(f).
$$
It makes also sense to define $g_1'(\omega)$ for any measurable subset $\omega$ of $M$, by setting
$$
\boxed{
g_1'(\omega) = \inf_{\nu\in\QL(M)} \nu(\omega)
}
$$

\begin{remark}\label{remsphere}
In contrast to Remark \ref{remg1}, we may have $g_1'(\omega)\neq g_1'(\overline{\omega})$ even for a Jordan set $\omega$. This is the case if one takes $M=\mathbb{S}^2$, the unit sphere in $\R^3$, and $\omega$ the open northern hemisphere. Indeed, the Dirac along the equator is a QL (see, e.g., \cite{JakobsonZelditch}) and thus $g_1'(\omega)=0$. But we have $g_1'(\overline{\omega})=1/2$ (the infimum is reached for any QL that is the Dirac along a great circle transverse to the equator).
\end{remark}

\begin{remark}\label{remg1g1'}
Given any $\nu\in\QL(M)$, there exists a sequence of $\lambda\rightarrow+\infty$ such that $\phi_\lambda^2\, dx_g\rightharpoonup\nu$, and it follows from the Portmanteau theorem (see Appendix \ref{app:Portmanteau}) that $\nu(\omega)\geq \lim_{\lambda\rightarrow+\infty} \int_\omega \phi_\lambda^2\, dx_g\geq g_1(\omega)$ for any closed subset $\omega$ of $M$.
Hence
$$
g_1(\omega) \leq g_1'(\omega)\qquad\forall\omega\subset M\ \textrm{closed},
$$
or, more generally, for every Borel subset $\omega$ of $M$ not charging any QL on $M$. Even more generally, we have
$$
g_1(a)\leq g_1'(a)
$$
for every bounded Borel function $a$ on $S^*M$ for which the $\mu$-measure of the set of discontinuities of $a$ is zero for every $\mu\in\QL(S^*M)$.
In particular the inequality holds true for any $a\in\mathcal{S}^0(M)$.
We refer to Lemma \ref{lem_ineq_g1} for this result.

The above inequality may be wrong without the specific assumption on $\omega$ or on $a$. Indeed, consider again the example given in Remark \ref{remsphere}: $M=\mathbb{S}^2$, $\omega$ is the open northern hemisphere, then $g_1(\omega)=g_1(\overline{\omega})=1/2$ (by symmetry arguments as in \cite{Lebeau_JEDP1992}), whereas $g_1'(\omega)=0$.

Finally, it is interesting to notice that, for $\omega\subset M$ open, if $g_1(\omega)=0$ then $g_1'(\omega)=0$ (see Lemma \ref{lem_ineq_g1}).
\end{remark}

\paragraph{Invariant measures.}
Let $\I(S^*M)$ be the set of probability Radon measures on $S^*M$ that are invariant under the geodesic flow. It is a compact set. It is well known that, by the Egorov theorem, we have $\QL(S^*M)\subset\I(S^*M)$.\footnote{Indeed, by the Egorov theorem (see also Remark \ref{remEgorov}), $a_t=a\circ\varphi_t$ is the principal symbol of $A_t=e^{-it\sqrt{\triangle}}\Op(a)e^{it\sqrt{\triangle}}$. Let $\mu\in\QL(S^*M)$. By definition, $\mu(a)$ is the limit of (some subsequence of) $\langle\Op(a)\phi_j,\phi_j\rangle$, hence $\mu(a\circ\varphi_t)=\lim\langle\Op(a)e^{it\sqrt{\triangle}}\phi_j,e^{it\sqrt{\triangle}}\phi_j\rangle=\mu(a)$ because $e^{it\sqrt{\triangle}}\phi_j=e^{it\lambda_j}\phi_j$.}
The converse inclusion is not true. However, it is known that, if $M$ has the spectral gap property\footnote{We say that $M$ has the spectral gap property if there exists $c>0$ such that $\vert\lambda-\mu\vert\geq c$ for any two {\em distinct} eigenvalues $\lambda$ and $\mu$ of $\sqrt{\triangle}$. This property allows multiplicity.}, then $M$ is Zoll (i.e., all its geodesics are periodic) and $\QL(S^*M)=\I(S^*M)$ (see \cite[Theorem 2 and Remark 3]{Macia_CPDE2008}).

Given any 
bounded Borel measurable function $a$ on $S^*M$, we define
$$
\boxed{
g_1''(a) = \inf_{\mu\in \I(S^*M)} \int_{S^*M}a\, d\mu  
}
$$
Since $\QL(S^*M)\subset\I(S^*M)$, we have
$$
g_1''(a)\leq g_1'(a)   
$$
for every bounded Borel measurable function $a$ on $S^*M$. As before, given any bounded Borel measurable function $f$ on $M$, the notation $g_1''(f)$ stands for $g_1''(f\circ\pi)$ without any ambiguity.

\begin{remark}\label{rem_KM}
Since the set of extremal points of $\I(S^*M)$ is the set of ergodic measures, in the definition of $g_1''$ we can replace $\I(S^*M)$ by the set of ergodic measures in the infimum.
\end{remark}

\begin{remark}\label{rem_Sigmund}
It follows from \cite{Sigmund_1972} that, if the manifold $M$ (which is connected and compact) is of negative curvature then the set of Dirac measures $\delta_\gamma$ along periodic geodesic rays $\gamma\in\Gamma$ is dense in $\I(S^*M)$ for the vague topology, and therefore
$$
g_1''(a) = \inf_{\gamma\in\Gamma\ \textrm{periodic}} \delta_\gamma(a)
= \inf \left\{ \frac{1}{T} \int_0^{T} a \circ \varphi_t(z)\, dt \ \mid\ z\in S^*M,\ T>0,\ \varphi_T(z)=z \right\}
$$
for every continuous function $a$ on $S^*M$.
\end{remark}

\subsection{Known results on Zoll manifolds}\label{sec:knownresults}
Recall that a Zoll manifold is a smooth connected closed Riemannian manifold without boundary, of which all geodesics are periodic. 
Thanks to a theorem by Wadsley (see \cite{Besse}), they have a least common period $T>0$. This does not mean that all geodesics are $T$-periodic: there may exist exceptional geodesics with period less than $T$, like in the lens-spaces, that are quotients of $\mathbb{S}^{2m-1}$ by certain finite cyclic groups of isometries.

Note that, in some of the existing literature, ``Zoll manifold" means that not only all geodesics are periodic, but also, have the same period. Here, we relax the latter statement (we could name this kind of manifold a ``weak Zoll manifold").

We consider the eigenvalues $\lambda$ of the operator $\sqrt{\triangle}$ considered on the compact manifold $M$.
Let $X$ be the Hamiltonian vector field on $S^*M$ of $\sqrt{\triangle}$. 
Note that $e^{tX}=\varphi_t$ for every $t\in\R$.
Denoting by $L_X$ the Lie derivative with respect to $X$, we define the self-adjoint operator $S=\frac{1}{i}L_X$. We also define $\Sigma$ as the set of closure points of all $\lambda-\mu$.

Let $A\in\Psi^0(M)$, of principal symbol $a$. For every function $f$ on $\R$, we set
$$
A_f = \int_\R \hat{f}(t) e^{-it\sqrt{\triangle}} A e^{it\sqrt{\triangle}} \,dt.
$$
Following \cite{Guillemin_Duke}, its principal symbol is computed on a finite time interval (by the Egorov theorem) and by passing to the limit (by using the Lebesgue dominated convergence theorem), and we get
$$
a_f = \sigma_P(A_f) = \int_\R \hat{f}(t) a \circ e^{tX} \, dt = \int_\R \hat{f}(t) (e^{tX})^*a \, dt .
$$
We will also denote $\varphi_t=e^{tX}$.

Besides, we have
$$
(e^{tX})^*a = a\circ e^{tX} = e^{tL_X}a = e^{itS}a ,
$$
and hence
$$
a_f = \sigma_P(A_f) = \int_\R \hat{f}(t) e^{itS}a \,dt = f(S)a.
$$
Denoting by $P_\lambda$ the projection onto the eigenspace corresponding to the eigenvalue $\lambda$, we have 
\begin{equation}\label{specdecomp}
\sqrt{\triangle} = \sum_{\lambda\in\mathrm{Spec}(\sqrt{\triangle})} \lambda P_\lambda, \qquad    e^{it\sqrt{\triangle}}=\sum_{\lambda\in\mathrm{Spec}(\sqrt{\triangle})} e^{it\lambda} P_\lambda,
\end{equation}
and we obtain
$$
A_f = \sum_{\lambda,\mu} f(\lambda-\mu) P_\lambda A P_\mu.
$$
Note that, by definition of $S$, using that $L_Xa=\{H,a\}$ where $H=\sigma_P(\sqrt{\triangle}) $ (Hamiltonian), we have $Sa=\frac{1}{i}\{a,H\}$.
As a consequence, the eigenfunctions of $S$ corresponding to the eigenvalue $0$ are exactly the functions that are invariant under the geodesic flow. 

It is remarkable that periodicity of geodesics and periodicity of the spectrum are closely related (see \cite{CdV,DuistermaatGuillemin_IM1975,Guillemin_Duke, Helton}). We gather these classical results in the following theorem.

\begin{theorem}[\cite{CdV,DuistermaatGuillemin_IM1975,Guillemin_Duke, Helton}]\label{thm_Guillemin}
We have the following results:
\begin{itemize}
\item $\mathrm{Spec}(S) \subset \Sigma$.
\item If there exists a non-periodic geodesic, then $\mathrm{Spec}(S) = \R$, and thus $\Sigma = \R$.
\item $M$ is Zoll if and only if $\Sigma\neq\R$. In this case, we have $\Sigma = \frac{2\pi}{T}\Z$, where $T$ is the smallest common period.
\end{itemize}
\end{theorem}

Note that, if $\Sigma=\R$, then there exists a non-periodic geodesic. Indeed, otherwise any geodesic would be periodic, and by the Wadsley theorem, $M$ would be Zoll and then this implies that $\Sigma = \frac{2\pi}{T}\Z$.

Actually, if $M$ is Zoll, then, denoting by $T$ the (common) period, we have $\lambda_j  \simeq  \frac{2\pi}{T} (\sigma+n_i)$ for some $n_i\in\Z$ (asymptotic spectrum of $\sqrt{\triangle}$).
In other words, this means that the eigenvalues cluster along the net $\frac{2\pi\sigma}{T} + \frac{2\pi}{T}\Z$.

\subsection{Main results: new characterizations of Zoll manifolds}\label{sec:mainresults}
Given a periodic ray $\gamma$ on $M$, the Dirac measure $\delta_\gamma$ on $M$ is defined by $\delta_\gamma(f)=\frac{1}{T}\int_0^Tf(\gamma(t))\, dt$ for every $f\in C^0(M)$, where $T$ is the period of $\gamma$.
Accordingly, given a periodic geodesic $\tilde\gamma$ on $S^*M$, the Dirac measure $\delta_{\tilde\gamma}$ on $S^*M$ is defined by $\delta_\gamma(a)=\frac{1}{T}\int_0^Ta\circ\varphi_t(z)\, dt = \bar a_T(z)$ for every $a\in C^0(S^*M)$, where $\tilde\gamma(t)=\varphi_t(z)$ for some $z\in S^*M$.

Note that, for a periodic geodesic $\tilde\gamma$ on $S^*M$, setting $\gamma=\pi\circ\tilde\gamma$, we have $\delta_\gamma\in\QL(M)$ if and only if $\delta_{\tilde\gamma}\in\QL(S^*M)$ (this follows from Proposition \ref{propQLM} in Appendix \ref{app:measures} and from the fact that $\QL(S^*M)\subset\I(S^*M)$).
Hence, in the theorem below, saying that ``the Dirac along any periodic ray is a QL on $M$" is equivalent to saying that ``the Dirac along any geodesic is a QL".

Before stating the result hereafter, we give two definitions concerning the spectrum:
\begin{itemize}
\item We say that $M$ has the \textit{spectral gap property} if there exists $c>0$ such that $\vert\lambda-\mu\vert\geq c$ for any two {\em distinct} eigenvalues $\lambda$ and $\mu$ of $\sqrt{\triangle}$. This property allows multiplicity.
\item We say that the spectrum is \textit{uniformly locally finite} if there exist $\ell>0$ and $m\in\N^*$ such that the intersection of the spectrum with any interval of length $\ell$ has at most $m$ {\em distinct} elements. This does not preclude multiplicity with arbitrarily large order.
\end{itemize}

Also, in order to appreciate the contents of the following result, it is useful to note that
$$
\boxed{
g_2^T(a)\leq g_2(a)\leq g_2'(a)\leq g_1''(a) \leq g_1'(a) 
}
$$
for every $T>0$ and for every bounded Borel measurable function $a$ on $S^*M$, and that
$$
\boxed{
\begin{array}{ll}
g_1(a) \leq g_1'(a)\  & \forall a\in C^0(S^*M)  \\
g_1(\omega) \leq g_1'(\omega)\  & \forall \omega\subset M\ \textrm{closed} 
\end{array}
}
$$
These facts will be proved in Lemmas \ref{leminegalites} and \ref{lem_ineq_g1}.

The next theorem, which is the main result of this paper, gives new characterizations of Zoll manifolds and relations with quantum limits.

\begin{theorem}\label{mainthm}
\ 
\begin{enumerate}
\item The following statements are equivalent:
\begin{itemize}
\item $M$ is Zoll and $\delta_\gamma\in\QL(M)$ for every periodic ray $\gamma\in\Gamma$;
\item $g_2(a)=g_2'(a)=g_1''(a)=g_1'(a)$ for every bounded Borel measurable function $a$ on $S^*M$;
\item $g_2'(\omega)=g_1''(\omega)=g_1'(\omega)$ for every $\omega\subset M$ Borel measurable;
\item there exists $T>0$ such that $g_1(\omega) \leq g_2^T(\omega)$ for any closed subset $\omega\subset M$.
\end{itemize}
Moreover, the smallest $T>0$ such that $g_1(\omega) \leq g_2^T(\omega)$ for any closed subset $\omega\subset M$ is the smallest period of geodesics of the Zoll manifold $M$.

\item The following statements are equivalent:
\begin{itemize}
\item $M$ is Zoll;
\item $g_2(a)=g_2'(a)$ for every bounded Borel measurable function $a$ on $S^*M$;
\item $g_2(\omega)=g_2'(\omega)$ for every $\omega\subset M$ Borel measurable;
\item $g_2'(a)=g_1''(a)$ for every bounded Borel measurable function $a$ on $S^*M$;
\item $g_2'(\omega)=g_1''(\omega)$ for every $\omega\subset M$ Borel measurable;
\item there exists $T>0$ such that $g_2^T(\omega)=g_2(\omega)$ for every open subset $\omega\subset M$;
\item there exists $T>0$ such that $g_2^T(\omega)=g_2(\omega)$ for every closed subset $\omega\subset M$;
\item there exists $T>0$ such that $g_2^T(\omega) = g_2'(\omega)$ for any closed subset $\omega\subset M$;
\item $g_2(\omega)=g'_2(\omega)$ for every closed subset $\omega\subset M$.
\end{itemize}
Moreover, the smallest $T>0$ such that the items above are satisfied is the smallest period of geodesics of the Zoll manifold $M$.

\item If $g_1(\omega)\leq g_1''(\omega)$ for any $\omega\subset M$ closed then $\delta_\gamma\in\QL(M)$ for every periodic ray $\gamma\in\Gamma$.
Moreover, for every minimally invariant\footnote{A minimally invariant set is a nonempty closed invariant set containing no proper closed invariant subset.} compact set $K$, there exists a quantum limit whose support is $K$.

\item If $M$ is Zoll and is a two-point homogeneous space\footnote{By definition, this means that, given points $x_0,x_1,y_0,y_1$ in $M$ such that $d(x_0,y_0)=d(x_1,y_1)$, there exists an isometry $\varphi$ of $M$ such that $\varphi(x_0)=x_1$ and $\varphi(y_0)=y_1$. This is equivalent to say that the group $\mathrm{Iso}(M)$ of isometries of $M$ acts transitively on $M\times M$.} and if the Dirac along any periodic geodesic is a QL then $\QL(S^*M)=\I(S^*M)$.

\item Under the spectral gap assumption, $M$ is Zoll and $\delta_\gamma\in\QL(M)$ for every periodic ray $\gamma\in\Gamma$.

\item If the spectrum is uniformly locally finite then $M$ is Zoll and for every ray $\gamma\in\Gamma$ there exists $\nu\in\QL(M)$ such that $\nu(\gamma(\R))>0$.

\end{enumerate}
\end{theorem}

\begin{remark}
The statements 3 and 4 of the theorem above are already known.
Statement 3 is exactly the contents of \cite[Theorems 4 and 8]{Macia_CPDE2008}. Concerning the statement 4: under the spectral gap assumption, we have $\Sigma\neq\R$ and thus $M$ is Zoll by Theorem \ref{thm_Guillemin}, and the fact that the Dirac along any geodesic is a QL has been established in \cite[Theorem 2 and Remark 3]{Macia_CPDE2008}.
\end{remark}

\begin{remark}
The assumption of uniform locally finite spectrum means that clustering is possible but only with a uniformly bounded number of {\em distinct} eigenvalues, but it allows arbitrary large multiplicities of eigenvalues.

Note that, by Theorem \ref{thm_Guillemin}, $M$ is Zoll if and only if $\Sigma\neq\R$. Therefore, in turn we have obtained that if the spectrum is uniformly locally finite then we must have $\Sigma\neq\R$.
This rules out the possibility of having a spectrum consisting, for instance, of all $j^{2/n},\ j^{2/n}+q_j$, for $j\in\N^*$, where $(q_j)_{j\in\N^*}$ is a countable description of $\mathbb{Q}$ (because then we have $\mathbb{Q}\subset\Sigma$, and hence $\Sigma=\R$).
\end{remark}

\begin{remark}
By the Egorov theorem, we have the inclusion $\QL(S^*M)\subset\I(S^*M)$, and in general this inclusion is strict. Remarks are in order:
\begin{itemize}
\item We have $\QL(S^*M)=\I(S^*M)$ when $M$ is the sphere in any dimension endowed with its canonical metric (see \cite{JakobsonZelditch}), or more generally, when $M$ is a compact rank-one symmetric space, which is a special case of a Zoll manifold (see \cite{Macia_CPDE2008}).
\item We do not know if, for a given Riemannian manifold $M$, the equality $\QL(S^*M)=\I(S^*M)$ implies that $M$ is Zoll.
\item Conversely, $M$ Zoll does not imply $\QL(S^*M)=\I(S^*M)$. 
Indeed, by \cite[Theorem 1.4]{MaciaRiviere_CMP2016}, there exist two-dimensional Zoll manifolds $M$ (Tannery surfaces) for which there exists a ray $\gamma$ (and even, an open set of rays) such that $\delta_\gamma\notin\QL(M)$. In particular, for such Zoll manifolds we have $\QL(S^*M)\subsetneq\I(S^*M)$. Also, by Theorem \ref{mainthm}, there must exist a Borel measurable subset $\omega\subset M$ (and even, an open subset) such that $g_1''(\omega) < g_1'(\omega)$.
\item We do not know any example of a Zoll manifold $M$ for which the spectrum is uniformly locally finite and $\QL(S^*M)\subsetneq\I(S^*M)$. 
\item It is interesting to note that if $M$ is of negative curvature then the Dirac $\delta_\gamma$ along a periodic ray $\gamma\in\Gamma$ can never be a QL (see \cite{Anantharaman,AnantharamanNonnenmacher}). 
\end{itemize}
\end{remark}

\section{Proofs}\label{sec:proofs}

\subsection{General results}
Note the obvious fact that if $a$ and $b$ are functions such that $a\leq b$ and for which the following quantities make sense, then $g_2^T(a)\leq g_2^T(b)$, $g_2(a)\leq g_2(b)$, $g_1(a)\leq g_1(b)$, $g_1'(a)\leq g_1'(b)$ and $g_1''(a)\leq g_1''(b)$. In other words, the functionals that we have defined are nondecreasing.

\subsubsection{Semi-continuity properties}
\begin{lemma}\label{lemBorel}
Let $\omega$ be a subset of $M$, let $T>0$ be arbitrary and let $(h_k)_{k\in\N^*}$ be a uniformly bounded sequence of Borel measurable functions on $M$.
\begin{itemize}
\item If $h_k$ converges pointwise to $\chi_\omega$, then
$$
\limsup_{k\rightarrow +\infty} g_2^T(h_k)\leq g_2^T(\omega), \qquad
\limsup_{k\rightarrow +\infty} g_1'(h_k)\leq g_1'(\omega) , \qquad
\limsup_{k\rightarrow +\infty} g_1''(h_k)\leq g_1''(\omega) , 
$$
and if moreover $\chi_{\omega}\leq h_k$ for every $k\in\N^*$, then
$$
g_2^T(\omega) = \lim_{k\rightarrow +\infty} g_2^T(h_k) = \inf_{k\in\N^*} g_2^T(h_k) ,
$$
$$
g_1'(\omega) = \lim_{k\rightarrow +\infty} g_1'(h_k) = \inf_{k\in\N^*} g_1'(h_k) ,\qquad
g_1''(\omega) = \lim_{k\rightarrow +\infty} g_1''(h_k) = \inf_{k\in\N^*} g_1''(h_k) .
$$
\item If $h_k$ converges Lebesgue almost everywhere to $\chi_\omega$, then
$$
\limsup_{k\rightarrow +\infty} g_1(h_k)\leq g_1(\omega) ,
$$
and if moreover $\chi_{\omega}\leq h_k$ for every $k\in\N^*$, then
$$
g_1(\omega) = \lim_{k\rightarrow +\infty} g_1(h_k) = \inf_{k\in\N^*} g_1(h_k) .
$$
\end{itemize}
\end{lemma}

\begin{proof}\ 
\begin{itemize}
\item Let $\gamma\in\Gamma$ be arbitrary. By pointwise convergence, we have  $h_k(\gamma(t))\rightarrow\chi_\omega(\gamma(t))$ for every $t\in[0,T]$, and it follows from the dominated convergence theorem that $g_2^T(h_k)\leq \frac{1}{T} \int_0^T h_k(\gamma(t))\, dt \rightarrow \frac{1}{T} \int_0^T \chi_\omega(\gamma(t))\, dt$, and thus $\limsup_{k\rightarrow +\infty} g_2^T(h_k)\leq \frac{1}{T} \int_0^T \chi_\omega(\gamma(t))\, dt$. Since this inequality is valid for any $\gamma\in\Gamma$, the first inequality follows.

By compactness of QLs, there exists $\nu\in\QL(M)$ such that $g_1'(\omega)=\nu(\omega)$. By dominated convergence, we have $g_1'(h_k)\leq \int_M h_k\, d\nu\rightarrow \int_M \chi_\omega\, d\nu=\nu(\omega)=g_1'(\omega)$, whence the second inequality. The proof for $g_1''$ is similar.

If moreover $\chi_{\omega}\leq h_k$ then $\limsup_{k\rightarrow+\infty} g_2^T(h_k)\leq g_2^T(\omega)\leq g_2^T(h_k)$ and $\limsup_{k\rightarrow+\infty} g_1'(h_k)\leq g_1'(\omega)\leq g_1'(h_k)$ and the result follows.

\item 
We have $g_1(\omega) = \inf_{\phi\in\mathcal{E}} \nu_\phi(\omega)$ with $\nu_\phi = \phi^2\, dx_g$. Let $\phi\in\mathcal{E}$ be arbitrary. For every $k\in\N^*$ we have $g_1(h_k) \leq \int_M h_k\phi^2\, dx_g$, and besides, by dominated convergence we have $\int_M h_k\phi^2\, dx_g \rightarrow \int_\omega \phi^2\, dx_g = \nu_\phi(\omega)$, hence $\limsup_{k\rightarrow+\infty} g_1(h_k)\leq \nu_\phi(\omega)$. Since $\phi\in\mathcal{E}$ is arbitrary, we get $\limsup_{k\rightarrow+\infty} g_1(h_k)\leq \inf_{\phi\in\mathcal{E}}\nu_\phi(\omega) = g_1(\omega)$.

If moreover $\chi_{\omega}\leq h_k$ then $\limsup_{k\rightarrow+\infty} g_1(h_k)\leq g_1(\omega)\leq g_1(h_k)$ and the result follows.
\end{itemize}
\end{proof}

Note that, in the proof for $g_2^T$ and $g_1'$, we use the fact that $h_k(x)\rightarrow\chi_\omega(x)$ for \textit{every} $x$. Almost everywhere convergence (in the Lebesgue sense) would not be enough.

\begin{remark}\label{rem_openclosed}
We denote by $d$ the geodesic distance on $(M,g)$. It is interesting to note that, given any subset $\omega$ of $M$:
\begin{itemize}
\item $\omega$ is open if and only if there exists a sequence of continuous functions $h_k$ on $M$ satisfying $0\leq h_k\leq h_{k+1}\leq\chi_\omega$ for every $k\in\N^*$ and converging pointwise to $\chi_\omega$.

Indeed, if $\omega$ is open, then one can take for instance $h_k(x) = \min(1, k\, d(x,\omega^c))$. Conversely, since $h_k(x)=0$ for every $x\in\omega^c$, by continuity of $h_k$ it follows that $h_k(x)=0$ for every $x\in\overline{\omega^c}=(\mathring{\omega})^c$. Now take $x\in\omega\setminus\mathring{\omega}$. We have $h_k(x)= 0$, and $h_k(x)\rightarrow\chi_\omega(x)$, hence $\chi_\omega(x)=0$ and therefore $x\in\omega^c$. Hence $\omega$ is open.

\item $\omega$ is closed if and only if there exists a sequence of continuous functions $h_k$ on $M$ satisfying $0\leq \chi_{\omega}\leq h_{k+1}\leq h_k\leq 1$ for every $k\in\N^*$ and converging pointwise to $\chi_\omega$.

Indeed, if $\omega$ is closed, then one can take $h_k(x) = \max(0,1-k\, d(x,\omega))$. Conversely, since $h_k(x)=1$ for every $x\in\omega$, by continuity of $h_k$ it follows that $h_k(x)=1$ for every $x\in\overline{\omega}$, and thus $\chi_{\overline{\omega}}\leq h_k\leq 1$. Now take $x\in\overline{\omega}\setminus\omega$. We have $h_k(x)=1$ and $h_k(x)\rightarrow\chi_\omega(x)$, hence $\chi_\omega(x)=1$ and therefore $x\in\omega$. Hence $\omega$ is closed.
\end{itemize}
\end{remark}

\begin{lemma}\label{lemopen}
Let $\omega$ be an open subset of $M$ and let $T>0$ be arbitrary. For every sequence of continuous functions $h_k$ on $M$ converging pointwise to $\chi_\omega$, satisfying moreover $0\leq h_k\leq h_{k+1}\leq\chi_\omega$ for every $k\in\N^*$, we have
\begin{equation*}
\begin{split}
g_2^T(\omega) &= \lim_{k\rightarrow +\infty} g_2^T(h_k) = \sup_{k\in\N^*} g_2^T(h_k) ,\qquad
g_2(\omega) = \lim_{k\rightarrow +\infty} g_2(h_k) = \sup_{k\in\N^*} g_2(h_k) ,\\
g_1'(\omega) &= \lim_{k\rightarrow +\infty} g_1'(h_k) = \sup_{k\in\N^*} g_1'(h_k), \qquad\ 
g_1''(\omega) = \lim_{k\rightarrow +\infty} g_1''(h_k) = \sup_{k\in\N^*} g_1''(h_k).
\end{split}
\end{equation*}
\end{lemma}

Note that, for $g_1$, the property
$g_1(\omega) = \lim_{k\rightarrow +\infty} g_1(h_k) = \sup_{k\in\N^*} g_1(h_k) $
may fail. Indeed, take $M=\mathbb{S}^2$, $\omega$ the open Northern hemisphere, then $g_1(\omega)=1/2$ and $g_1(h_k)=0$ for every $k$.

\begin{proof}
Since $h_k\leq\chi_{\omega}$, we have $g_2^T(h_k)\leq g_2^T(\omega)$.
By continuity of $h_k$ and by compactness of geodesics, there exists a ray $\gamma_k$ such that $g_2^T(h_k) = \frac{1}{T} \int_0^T h_k(\gamma_k(t))\, dt$. Again by compactness of geodesics, up to some subsequence $\gamma_k$ converges to a ray $\bar\gamma$ in $C^0([0,T],M)$. We claim that
$$
\liminf_{k\rightarrow+\infty} h_k(\gamma_k(t))\geq \chi_\omega(\bar\gamma(t))\qquad\forall t\in[0,T].
$$
Indeed, either $\bar\gamma(t)\notin\omega$ and then $\chi_\omega(\bar\gamma(t))=0$ and  the inequality is obviously satisfied, or $\bar\gamma(t)\in\omega$ and then, using that $\omega$ is open, for $k$ large enough we have $\gamma_k(t)\in U$ where $U\subset\omega$ is a compact neighborhood of $\bar\gamma(t)$. Since $h_k$ is monotonically nondecreasing and $\chi_\omega$ is continuous on $U$, it follows from the Dini theorem that $h_k$ converges uniformly to $\chi_\omega$ on $U$, and then we infer that $h_k(\gamma_k(t))\rightarrow 1=\chi_\omega(\bar\gamma(t))$. The claim is proved.
Now, we infer from the Fatou lemma that
\begin{multline*}
g_2^T(\omega) \leq \frac{1}{T}\int_0^T\chi_\omega(\bar\gamma(t))\, dt 
\leq \frac{1}{T}\int_0^T \liminf_{k\rightarrow+\infty} h_k(\gamma_k(t)) \, dt \\
\leq \liminf_{k\rightarrow+\infty} \frac{1}{T}\int_0^T h_k(\gamma_k(t)) \, dt
= \liminf_{k\rightarrow+\infty} g_2^T(h_k)
\leq g_2^T(\omega)
\end{multline*}
and we get the equality.

Since $g_2(\omega)=\sup_{T>0} g_2^T(\omega)$ by Remark \ref{rem_g_2_sup}, interverting the $\sup$ yields
$$
g_2(\omega) = \sup_{T>0} \sup_{k\in\N^*} g_2^T(h_k) = \sup_{k\in\N^*} \sup_{T>0} g_2^T(h_k) = \sup_{k\in\N^*} g_2(h_k).
$$

By compactness of QLs, there exists $\nu_k\in \QL(M)$ such that $g_1'(h_k) = \nu_k(h_k)$. Again by compactness of QLs, up to some subsequence we have $\nu_k\rightharpoonup\bar\nu\in \QL(M)$. Since $\omega$ is open, we can write
$$
\omega = \cup_{\varepsilon>0} V_\varepsilon,\qquad V_\varepsilon = \{ x\in\omega\ \mid\ d(x,\omega^c)>\varepsilonÊ\} ,
$$
$V_\varepsilon$ being open.
Let $\varepsilon>0$ be arbitrarily fixed. By construction, $\overline{V_\varepsilon} = \{ x\in\omega\ \mid\ d(x,\omega^c)\geq \varepsilonÊ\}$ is a compact set contained in the open set $\omega$. Since $h_k$ converges monotonically pointwise to $\chi_\omega$, by the Dini theorem, $h_k$ converges uniformly to $1$ on $\overline{V_\varepsilon}$, and thus without loss of generality we write that $\chi_{V_\varepsilon}\leq h_k$.
By the Portmanteau theorem (see Appendix \ref{app:Portmanteau}), we have
$\bar\nu(V_\varepsilon)\leq \liminf_{k\rightarrow+\infty} \nu_k(V_\varepsilon)$, and we have $\nu_k(V_\varepsilon) = \int_M \chi_{V_\varepsilon}\, d\nu_k \leq \int_M h_k\, d\nu_k = g_1'(h_k)$. It follows that $\bar\nu(V_\varepsilon)\leq \liminf_{k\rightarrow+\infty} g_1'(h_k)$.
Now we let $\varepsilon$ converge to $0$ and we get that $g_1'(\omega)\leq \bar\nu(\omega)\leq \liminf_{k\rightarrow+\infty} g_1'(h_k)$. The result for $g_1'$ follows. The proof for $g_1''$ is similar.
\end{proof}

\begin{remark}
The results of Lemmas \ref{lemBorel} and \ref{lemopen} are valid as well for subsets of $S^*M$ (which is a metric space).
\end{remark}

\begin{lemma}\label{lemg2open}
Let $\omega$ be an open subset of $M$ and let $T>0$ be arbitrary.
There exists $\gamma\in\Gamma$ such that $g_2^T(\omega) = \frac{1}{T}\int_0^T\chi_\omega(\gamma(t))\, dt$, i.e., the infimum in the definition of $g_2^T(\omega)$ is reached.
\end{lemma}

\begin{proof}
The argument is almost contained in the proof of Lemma \ref{lemopen}, but for completeness we give the detail.
Let $(\gamma_k)_{k\in\N^*}$ be a sequence of rays such that $\frac{1}{T}\int_0^T\chi_\omega(\gamma_k(t))\, dt \rightarrow g_2^T(\omega)$. By compactness of geodesics, $\gamma_k(\cdot)$ converges uniformly to some ray $\gamma(\cdot)$ on $[0,T]$.

Let $t\in[0,T]$ be arbitrary. If $\gamma(t)\in \omega$ then for $k$ large enough we have $\gamma_k(t)\in \omega$, and thus $1=\chi_{\omega}(\gamma(t))\leq \chi_{\omega}(\gamma_k(t))=1$. If $\gamma(t)\in M\setminus\omega$ then $0=\chi_{\omega}(\gamma(t))\leq \chi_{\omega}(\gamma_k(t))$ for any $k$. In all cases, we have obtained the inequality $\chi_{\omega}(\gamma(t)) \leq \liminf_{k\rightarrow+\infty} \chi_{\omega}(\gamma_k(t))$, for every $t\in[0,T]$.

By the Fatou lemma, we infer that
$$
g_2^T(\omega)Ê\leq \frac{1}{T}\int_0^T\chi_\omega(\gamma(t))\, dt \leq \frac{1}{T}\int_0^T \liminf_{k\rightarrow+\infty}\chi_\omega(\gamma_k(t))\, dt 
\leq \liminf_{k\rightarrow+\infty} \frac{1}{T}\int_0^T \chi_\omega(\gamma_k(t))\, dt = g_2^T(\omega)
$$
and the equality follows.
\end{proof}

\subsubsection{General inequalities}

\begin{lemma}\label{leminegalites}
We have
$$
g_2^T(a)\leq g_2(a)\leq g_2'(a)\leq g_1''(a) \leq g_1'(a) 
$$
for every $T>0$ and for every bounded Borel measurable function $a$ on $S^*M$.
\end{lemma}

\begin{proof}
Given any $\mu\in\I(S^*M)$, we have $\int_{S^*M}a\, d\mu = \int_{S^*M} a\circ\phi_t\, d\mu$ for any bounded Borel measurable function $a$ on $S^*M$ and for any $t\in\R$, and thus $\int_{S^*M}a\, d\mu = \int_{S^*M} \frac{1}{T} \int_0^T a\circ\phi_t\, d\mu = \int_{S^*M} \bar a_T\, d\mu$ for any $T>0$. Passing to the limit we get $\int_{S^*M}a\, d\mu = \int_{S^*M} \liminf_{T\rightarrow +\infty} \bar a_T\, d\mu \geq g_2'(a)$, because $g_2'(a) = \inf \liminf_{T\rightarrow +\infty} \bar a_T$ by definition. The result follows.
\end{proof}

\begin{lemma}\label{lem_ineq_g1}
\begin{itemize}
\item We have $g_1(a)\leq g_1'(a)$ for every bounded Borel function $a$ on $S^*M$ for which the $\mu$-measure of the set of discontinuities of $a$ is zero for every $\mu\in\QL(S^*M)$. In particular the inequality is valid for every continuous function $a$ on $S^*M$.
\item Given any closed subset $\omega$ of $M$ (or of $S^*M$), we have $g_1(\omega)\leq g_1'(\omega)$.
\item Let $\omega$ be an open subset of $M$. If $g_1(\omega)=0$ then $g_1'(\omega)=0$. 
\end{itemize}
\end{lemma}

\begin{proof}
The first claim follows by the Portmanteau theorem (see Appendix \ref{app:Portmanteau}) and by definition of QLs.

The second claim follows by using Lemma \ref{lemBorel} and Remark \ref{rem_openclosed}.

Let us prove the third item.
If $g_1(\omega)=0$ for some measurable subset $\omega$ of positive Lebesgue measure, then either $\int_\omega\phi^2\, dx_g=0$ for some $\phi\in\mathcal{E}$ or $\liminf_{\lambda\rightarrow+\infty}\int_\omega\phi_\lambda^2\, dx_g=0$. The first possibility cannot occur: it cannot happen that $\int_\omega\phi^2\, dx_g=0$ because this would imply that $\phi=0$ on a subset $\omega$ of positive Lebesgue measure and thus of Hausdorff dimension $n$, which is impossible by results of \cite{DonnellyFefferman,HardtSimon,Lin} (this impossibility is obvious when the manifold $M$ is analytic  because then $\phi$ is analytic). Therefore $\liminf_{\lambda\rightarrow+\infty}\int_\omega\phi_\lambda^2\, dx_g=0$. Up to some subsequence, there exists $\nu\in\QL(M)$ that is the weak limit of $\phi_\lambda^2\, dx_g$. Now, by the Portmanteau theorem (see Appendix \ref{app:Portmanteau}), since $\omega$ is open we have $\nu(\omega)\leq \liminf_{\lambda\rightarrow+\infty}\int_\omega\phi_\lambda^2\, dx_g$, and thus $\nu(\omega)=0$. The claim follows.
\end{proof}

\begin{lemma}\label{lemg2g2prime}
Let $\omega$ be an open subset of $M$. If $g_2(\omega)=0$ then $g_2'(\omega)=g_1''(\omega)=0$.
\end{lemma}

\begin{proof}
By Remark \ref{rem_g_2_sup}, we have $g_2^T(\omega)\leq g_2(\omega)$ and thus $g_2^T(\omega)=0$ for every $T>0$. Now, by Lemma \ref{lemg2open}, for every $k\in\N^*$ there exists $\gamma_k\in\Gamma$ such that $g_2^k(\omega)=\frac{1}{k}\int_0^k\chi_\omega(\gamma_k(t))\, dt=0$, hence $\chi_\omega(\gamma_k(t))=0$ (i.e., $\gamma_k(t)\in M\setminus\omega$) for almost every $t\in[0,k]$.
By compactness of geodesics, up to some subsequence $\gamma_k$ converges to a ray $\gamma\in\Gamma$ uniformly on any compact interval. 
Using the fact that $M\setminus\omega$ is closed, 
we infer that $\overline{\gamma(\R)}\subset M\setminus\omega$. Not only it follows that $g_2'(\omega)=0$, but also that $g_1''(\omega)=0$. To obtain the last statement, take an invariant probability measure $\nu$ on the compact set $\overline{\gamma(\R)}$ (there always exists at least one such measure) and consider the invariant probability measure $\nu_\gamma$ on $M$ defined by $\nu_\gamma(E)=\nu(E\cap\overline{\gamma(\R)})$ for every Borel set $E\subset M$.
\end{proof}

\subsection{Proof of Theorem \ref{mainthm}}

\begin{lemma}\label{lemZollg2}
If $M$ is Zoll then $g_2(a)=g_2'(a)=g_1''(a)$ for any bounded Borel measurable function $a$ on $S^*M$.

If moreover $\delta_\gamma\in\QL(M)$ for every periodic ray $\gamma\in\Gamma$ then $g_2(a)=g_2'(a)=g_1''(a)=g_1'(a)$ for any bounded Borel measurable function $a$ on $S^*M$.
\end{lemma}

\begin{proof}
Since $M$ is Zoll, all geodesics are periodic with a common period $T$ (by Wadsley's theorem) and hence, given a bounded Borel measurable function $a$ on $S^*M$, we have $\lim_{S\rightarrow+\infty}\frac{1}{S}\int_0^S a\circ\varphi_t\, dt = \frac{1}{T}\int_0^T a\circ\varphi_t\, dt$ and the limit is uniform on $S^*M$. Hence
$$
g_2(a)=g_2^T(a) = g_2'(a) = \inf_{z\in S^*M} \frac{1}{T}\int_0^{T} a\circ\varphi_t(z)\, dt = \inf_{\gamma\in\Gamma} \delta_\gamma(a) 
$$
for any bounded Borel measurable function $a$ on $S^*M$.
Since the Dirac $\delta_\gamma$ along any (periodic) geodesic is invariant, we have $g_1''(a) \leq \inf_{\gamma\in\Gamma} \delta_\gamma(a)$, and the equality $g_2(a)=g_2'(a)=g_1''(a)$ follows by Lemma \ref{leminegalites}.

If moreover the Dirac along any geodesic is a QL then $g_1'(a) \leq \inf_{\gamma\in\Gamma} \delta_\gamma(a)$, and the equality $g_2(a)=g_2'(a)=g_1''(a)=g_1'(a)$ follows by Lemma \ref{leminegalites} as well.
\end{proof}

\begin{lemma}
If $g_2'(\omega)=g_1''(\omega)$ for any Borel measurable subset $\omega$ of $M$ then $M$ is Zoll.

If $g_2'(\omega)=g_1''(\omega)=g_1'(\omega)$ for any Borel measurable subset $\omega$ of $M$ then $M$ is Zoll and $\delta_\gamma\in\QL(M)$ for every periodic ray $\gamma\in\Gamma$.
\end{lemma}

\begin{proof}
Let $\gamma$ be a ray. The subset $\omega = M\setminus\gamma(\R)$ of $M$ is Borel measurable because $\gamma(\R) = \cup_{n\in\N^*}\gamma([-n,n])$ is a countable union of closed sets. Obviously, we have $g'_2(\omega)=0$ and hence, using the assumption, $g_1''(\omega)=0$. This means that there exists an invariant measure $\nu$ such that $\nu(M\setminus\gamma(\R))=0$. Hence $\nu$ is concentrated on $\gamma(\R)$ and $\nu(\gamma(\R))=1$. It follows that $\nu(\gamma([-n,n]))\rightarrow 1$ as $n\rightarrow+\infty$.
By invariance we get that $\gamma$ must be periodic and that $\nu=\delta_\gamma$.

If moreover $g_2'(\omega)=g_1''(\omega)=g_1'(\omega)$, then we have $g_1'(\omega)=0$ and actually $\nu\in\QL(M)$ in the reasoning above. The lemma follows (see also Proposition \ref{propQLM} in Appendix \ref{app:measures}).
\end{proof}

\begin{lemma}\label{lem9zoll}
If $g_2(\omega)=g_2'(\omega)$ for any Borel measurable subset $\omega$ of $M$ then $M$ is Zoll.
\end{lemma}

\begin{proof}
Assume that $M$ is not Zoll and take a non-periodic ray $\gamma$.
We consider the Borel measurable set $\omega = M\setminus \cup_{k\in\N^*} \gamma([2^k,2^k+k])$. Then we have $g_2^{k}(\omega)=0$ (take the ray $\gamma$ restricted to the time interval $[2^k,2^k+k]$ of length $k$) and thus $g_2(\omega)=0$ while $g_2'(\omega)=1$. The latter equality is because 
$$
\lim_{k\rightarrow+\infty} \frac{1}{2^k}\int_0^{2^k} \chi_\omega(\gamma(t))\, dt = \lim_{k\rightarrow+\infty} \left( 1-\frac{1}{2^k}\sum_{j=1}^{k-1} j \right) = 1 .
$$
The lemma is proved.
\end{proof}

\begin{lemma}\label{lem_g2nondecreasing}
Given any $T>0$ and any bounded Borel measurable function $a$ on $S^*M$, the sequence $(g_2^{2^k T}(a))_{k\in\N^*}$ is nondecreasing (and converges to $g_2(a)$ as $T\rightarrow+\infty$).
\end{lemma}

\begin{proof}
Given any $\delta>1$, we have
\begin{multline*}
g_2^{\delta T}(a) = \inf \frac{1}{\delta T}\int_0^{\delta T} a\circ\varphi_t\, dt
= \inf \frac{1}{\delta T} \left( \int_0^{\frac{\delta}{2}T} a\circ\varphi_t\, dt + \int_{\frac{\delta}{2}T}^{\delta T} a\circ\varphi_t\, dt \right) \\
= \inf \frac{1}{\delta T} \left( \int_0^{\frac{\delta}{2}T} a\circ\varphi_t\, dt + \int_0^{\frac{\delta}{2}T} a\circ\varphi_t\, dt\, \circ\varphi_{\frac{\delta}{2}T} \right)
=\inf\left( \frac{1}{2} \bar a_{\frac{\delta}{2}T} + \frac{1}{2} \bar a_{\frac{\delta}{2}T}\circ\varphi_{\frac{\delta}{2}T} \right)\\
\geq \frac{1}{2} g_2^{\frac{\delta}{2}T}(a) + \frac{1}{2} g_2^{\frac{\delta}{2}T}(a) = g_2^{\frac{\delta}{2}T}(a) .
\end{multline*}
The lemma follows.
\end{proof}

\begin{lemma}\label{lemg2T}
$M$ is Zoll if and only if there exists $T>0$ such that $g_2(\omega)=g_2^T(\omega)$ for every open subset $\omega\subset M$, if and only if there exists $T>0$ such that $g_2(\omega)=g_2^T(\omega)$ for every closed subset $\omega\subset M$. Moreover the smallest period of geodesics is the smallest $T$ such that $g_2(\omega)=g_2^T(\omega)$ for every open (or closed) subset $\omega\subset M$.
\end{lemma}

\begin{proof}
If $M$ is Zoll with period $T$ then $g_2(\omega)=g_2^T(\omega)(=g_2'(\omega))$ for any $\omega$: this has been proved in Lemma \ref{lemZollg2}.

Conversely, if $g_2(\omega)=g_2^T(\omega)$ for some $T>0$ for every $\omega$ open, then, since by Lemma \ref{lem_g2nondecreasing} the sequence $(g_2^{2^k T}(\omega))_{k\in\N^*}$ is nondecreasing, we get $g_2^T(\omega)=g_2^{2^k T}(\omega)$ for any $\omega$ open and any $k$. Fix a ray $\gamma$ and take $\omega=M\setminus\gamma([0,T])$. Then $g_2^T(\omega)=0$, hence $g_2^{2^k T}(\omega)=0=\inf_{\gamma'\in\Gamma}\frac{1}{2^kT}\int_0^{2^kT}\chi_\omega(\gamma'(t))\, dt = \frac{1}{2^kT}\int_0^{2^kT}\chi_\omega(\gamma(t))\, dt$ (because this infimum is clearly reached for the ray $\gamma$) and thus $\gamma(t)\in\gamma([0,T])$ for any $t$. This implies that $\gamma$ is periodic.

Now, if the equality is valid on closed subsets, take $\omega_k = \{ x\in M\ \mid\ d(x,\gamma([0,T]))\geq \frac{1}{k}Ê\}$. Reasoning as above, we get $\inf_{\gamma'\in\Gamma}\frac{1}{2^kT}\int_0^{2^kT}\chi_{\omega_k}(\gamma'(t))\, dt = 0$, hence by compactness of rays there exists $\gamma_k\in\Gamma$ such that $\chi_{\omega_k}(\gamma_k(t))=0$ for every $t\in[0,2^kT]$, i.e., $d(\gamma_k(t),\gamma([0,T]))\leq\frac{1}{k}$ for every $t\in[0,2^kT]$. Passing to the limit, we get a limit ray $\gamma_\infty\in\Gamma$ such that $\gamma_\infty(\R)\subset\gamma([0,T])$. This implies that $\gamma_\infty=\gamma$ is periodic.
\end{proof}

\begin{lemma}\label{lem12Zoll}
$M$ is Zoll and $\delta_\gamma\in\QL(M)$ for every periodic ray $\gamma\in\Gamma$ if and only if there exists $T>0$ such that $g_1(\omega) \leq g_2^T(\omega)$ for any $\omega\subset M$ closed.
\end{lemma}

\begin{proof}
Assume that $M$ is Zoll and that $\delta_\gamma\in\QL(M)$ for every periodic ray $\gamma\in\Gamma$. Since all geodesics are periodic, they have the same period or a multiple of that period (Wadsley's Theorem, see \cite{Besse}), denoted by $T$. By Lemma \ref{lemZollg2} and Lemma \ref{lemg2T} we have $g_1'(\omega)=g_2^T(\omega)$ for any $\omega\subset M$ closed, and by Remark \ref{remg1g1'} we have $g_1(\omega)\leq g_1'(\omega)$, hence $g_1(\omega) \leq g_2^T(\omega)$.

Conversely, assume that there exists $T>0$ such that $g_1(\omega) \leq g_2^T(\omega)$ for any $\omega\subset M$ closed. Let $\gamma\in\Gamma$ be arbitrary. We define $\omega_\varepsilon^T = \{ x\in M\ \mid\ d(x,\gamma([0,T]))>\varepsilon\}$ and noting that $\omega_{\varepsilon_1}^T\subset\overline{\omega}_{\varepsilon_2}^T\subset \omega_{\varepsilon_3}^T$ when $\varepsilon_1>\varepsilon_2>\varepsilon_3$, the open set $M\setminus\gamma([0,T])$ can be written as the increasing union of open or of closed subsets:
$M\setminus\gamma([0,T]) = \cup_{\varepsilon>0} \omega_\varepsilon^T = \cup_{\varepsilon>0} \overline{\omega}_\varepsilon^T$. 
Therefore, there exist a sequence of open subsets $O_k$ of $M$, of closed subsets $F_k$ and of continuous functions $h_k$ on $M$, satisfying
\begin{equation}\label{nested}
\chi_{F_k}\leq h_k\leq \chi_{O_k}\leq \chi_{F_{k+1}} \leq \chi_{M\setminus\gamma([0,T])}
\end{equation}
for every $k\in\N^*$. This means that $\chi_{M\setminus\gamma([0,T])}$ is the pointwise limit of the increasing sequences $\chi_{F_k}$, $\chi_{O_k}$ and $h_k$.
Now, since $F_k\subset M\setminus\gamma([0,T])$, we have $g_2^T(F_k)=0$ and thus, by assumption, $g_1(F_k)=0$ for any $k\in\N^*$. Using \eqref{nested}, it follows that $g_1(O_k)=0$ for any $k\in\N^*$. Since $O_k$ has positive measure (at least, for $k$ large enough), we infer from the third item of Lemma \ref{lem_ineq_g1} that $g_1'(O_k)=0$ and thus, using again \eqref{nested}, $g_1'(h_k)=0$ for any $k$ large enough. Applying Lemma \ref{lemopen}, we get that $g_1'(M\setminus\gamma([0,T]))=0$. This means that there exists $\nu\in\QL(M)$ such that $\nu(M\setminus\gamma([0,T]))=0$. Hence $\nu$ is concentrated on $\gamma([0,T])$, and by invariance (see Proposition \ref{propQLM} in Appendix \ref{app:measures}) we infer two things: first, since the mass of $\nu$ is finite, the ray $\gamma$ must be periodic; second, we must have $\nu=\delta_\gamma$.
The lemma is proved.
\end{proof}

\begin{lemma}\label{lemg1g1''}
If $g_1(\omega)\leq g_1''(\omega)$ for any $\omega\subset M$ closed then $\delta_\gamma\in\QL(M)$ for every periodic ray $\gamma\in\Gamma$. Moreover, for every minimally invariant compact set $K$, there exists a quantum limit whose support is $K$.
\end{lemma}

\begin{proof}
Assume that $g_1(\omega)\leq g_1''(\omega)$ for any $\omega\subset M$ closed. Let $\gamma\in\Gamma$ be a periodic ray, of period $T$. As in the second part of the proof of Lemma \ref{lem12Zoll} above, we define the sets $\omega_\varepsilon^T$ and the sets $O_k$ and $F_k$ satisfying \eqref{nested}. Since the measure $\delta_\gamma$ is invariant under the geodesic flow, we have $g_1''(F_k)=0$ and thus $g_1(F_k)=0$ by assumption. As in the proof of Lemma \ref{lem12Zoll}, we first infer that $g_1'(M\setminus\gamma([0,T]))=0$ and then that $\delta_\gamma\in\QL(M)$.

Let us now prove the last statement. Let $K$ be a minimally invariant compact set. Noting that $K$ is invariant under the flow $(\varphi_t)_{t\in\R}$, there always exists at least one invariant measure $\mu$ on $K$. We consider the measure $\mu_K$ on $M$ defined as the restriction $\mu_K(E)=\mu(E\cap K)$ for every Borel set $E\subset M$. By construction, we have $\mu_K(M\setminus K)=0$ and the measure $\mu_K$ is invariant. Now, we define $\omega_\varepsilon = \{x\in M\ \mid\ d(x,K)>\varepsilon\}$, we have $M\setminus K = \cup_{\varepsilon>0}\omega_\varepsilon = \cup_{\varepsilon>0}\overline{\omega}_\varepsilon$, and we perform the same construction as in Lemma \ref{lem12Zoll}, with the sets $O_k$ and $F_k$ satisfying \eqref{nested}. We have $g_1''(F_k)=0$ and thus $g_1(F_k)=0$ by assumption. We then infer in the same way that $g_1'(M\setminus K)=0$, hence there exists $\nu\in\QL(M)$ such that $\nu(M\setminus K)=0$, i.e., $\nu$ is concentrated on $K$, and by invariance and minimality of $K$ we have $\supp(\nu)=K$.
\end{proof}

\begin{lemma}\label{crucial_lemma}
Let $\gamma\in\Gamma$ be a non-periodic ray, let $T>0$ be arbitrary and let $(U_k)_{k\in\N^*}$ be a decreasing sequence of open sets such that $\cap_{k\in\N^*} U_k = \gamma([0,T])$. Defining the closed subset $\omega_k= M \setminus U_k$, we have $g_2^T(\omega_k)=0$ and
$$
\lim_{k\rightarrow+\infty} g_2'(\omega_k) = 1. 
$$
\end{lemma}

\begin{proof}
We proceed by contradiction. Let us assume that $\liminf_{k\rightarrow+\infty} g_2'(\omega_k) <1$. Then there exists $c>0$ and a ray $\gamma'\in\Gamma$ such that
$$
\liminf_{T \to +\infty} \frac{1}{T} \int_0^T \chi_{\omega_k} (\gamma'(t)) \, dt \leq 1- c
$$
for every $k\in\N^*$, and hence there exists a sequence $(T_k)_{k\in\N^*}$ converging to $+\infty$ such that
$$
\frac{1}{T_k} \int_0^{T_k} \chi_{U_k} (\gamma'(t))\, dt \geq \frac{c}{2}
$$
for every $k$ large enough.

Let $T' >0$ be fixed. Let us prove now that there exists a sequence of rays $\tilde\gamma_k\in\Gamma$ such that 
\begin{eqnarray} \label{Uep2}
 \frac{1}{T'} \int_0^{T'} \chi_{\omega_k} (\tilde\gamma_k(t))\, dt \geq \frac{c}{4}
\end{eqnarray}
for every $k$ large enough. In what follows the notation $\lfloor\ \rfloor$ stands for the usual floor function.
Given any $k$ large enough, we have 
\begin{multline*} 
\frac{c}{2} \leq  \frac{1}{T_k} \int_0^{T_k} \chi_{U_k} (\gamma'(t))\, dt 
\leq \frac{T'}{T_k} \sum_{j=0}^{\left\lfloor \frac{T_k}{T'} \right\rfloor -1}   \frac{1}{T'} \int_{jT'}^{(j+1)T'} \chi_{U_k} (\gamma'(t))\, dt + \frac{1}{T_k}  \int_{\left\lfloor \frac{T_k}{T'} \right\rfloor T' }^{T_k} \chi_{U_k} (\gamma'(t))\, dt \\
\leq \frac{T'}{T_k} \sum_{j=0}^{\left\lfloor \frac{T_k}{T'} \right\rfloor -1}   \frac{1}{T'} \int_{0}^{T'} \chi_{U_k} (\gamma'(jT'+t))\, dt + \mathrm{o}(1)
\end{multline*}
where, to get the latter inequality and in particular the remainder term $\mathrm{o}(1)$ as $k\rightarrow+\infty$, we have bounded the last integral by $\frac{T'}{T_k} = \mathrm{o}(1)$, using the fact that $T_k - \left\lfloor \frac{T_k}{T'} \right\rfloor T'  \leq T'$. 
If $\frac{1}{T'} \int_{0}^{T'} \chi_{U_k} (\gamma'(jT'+t))\, dt < \frac{c}{4}$ for $j = 0,\ldots,\left\lfloor \frac{T_k}{T'} \right\rfloor-1$ then $\frac{c}{2} \leq \frac{T'}{T_k} \left\lfloor \frac{T_k}{T'} \right\rfloor \frac{c}{4} + \mathrm{o}(1)$, and since $T_k\rightarrow+\infty$ the right-hand side tends to $\frac{c}{4}$, yielding a contradiction. Hence 
$$ 
\exists j_k\inÊ\left\{ 0,\ldots, \left\lfloor \frac{T_k}{T'} \right\rfloor-1 \right\} \quad\textrm{s.t.}\quad
\frac{1}{T'} \int_{0}^{T'} \chi_{U_k} (\gamma'(j_kT'+t))\, dt\geq \frac{c}{4}
$$ 
and \eqref{Uep2} follows by taking $\tilde\gamma_k(\cdot)= \gamma'(j_kT'+\cdot)$.

By compactness of geodesics, up to some subsequence $\tilde\gamma_k$ converges uniformly on $[0,T']$ to some $\tilde\gamma\in\Gamma$. Setting
$$
A = \{ t \in [0,T'] \ \mid\ \tilde\gamma(t) \in \gamma([0,T]) \} 
$$
and noting that the Lebesgue measure of $A$ is at most equal to $T$, we have
\begin{multline*}
\frac{c}{4}\leq \frac{1}{T'} \int_{0}^{T'} \chi_{U_k} (\gamma'(j_kT'+t))\, dt = \frac{1}{T'} \int_{A} \chi_{U_k} (\tilde\gamma_k(t))\, dt + \frac{1}{T'} \int_{[0,T'] \setminus A} \chi_{U_k} (\tilde\gamma_k(t)) \, dt \\
\leq \frac{T}{T'} + \frac{1}{T'} \int_{[0,T'] \setminus A} \chi_{U_k} (\tilde\gamma_k(t)) \, dt 
\end{multline*}
By definition, if $t \in [0,T'] \setminus A$ then $\tilde\gamma(t) \notin \gamma([0,T])$ and hence $\lim_{k\rightarrow +\infty} \chi_{U_k}(\tilde\gamma_k(t))=0$ since $\cap_{k\in\N^*} U_k = \gamma([0,T])$. Therefore, by dominated convergence, the latter integral at the right-hand side of the above inequality converges to $0$ as $k\rightarrow +\infty$.
We obtain a contradiction if $T'$ is large enough. 
\end{proof}

\begin{lemma}
$M$ is Zoll if and only if there exists $T>0$ such that $g_2^T(\omega) = g_2'(\omega)$ for any $\omega\subset M$ closed.
\end{lemma}

\begin{proof}
If $M$ is Zoll then $g_2(\omega)=g_2'(\omega)$ for any $\omega\subset M$ closed by Lemma \ref{lemZollg2}, and by Lemma \ref{lemg2T} we have $g_2(\omega)=g_2^T(\omega)$.

Conversely, if there exists a non-periodic ray $\gamma\in\Gamma$, then by Lemma \ref{crucial_lemma} there exists $\omega^T\subset M$ closed such that $g_2^T(\omega^T)=0$ and $g_2'(\omega^T)\geq 1/2$. The lemma follows.
\end{proof}

\begin{lemma}\label{crucial_lemma_extended}
Let $\gamma\in\Gamma$ be a non-periodic ray. There exists a decreasing family $(\omega^\varepsilon)_{\varepsilon>0}$ of closed subsets of $M$, satisfying 
\begin{equation}\label{strictinclusions}
0<\varepsilon_1<\varepsilon_2\ \Rightarrow\ M\setminus\omega^{\varepsilon_1}\ \subset\  M\setminus\omega^{\varepsilon_2}
\end{equation}
for any positive real numbers $\varepsilon_1$ and $\varepsilon_2$, such that
$$
\bigcup_{k\in\N^*} \gamma([2^k,2^k+k])\subset M\setminus\omega^\varepsilon,\qquad
\vert M\setminus\omega^\varepsilon\vert\leq\varepsilon,\qquad
g_2(\omega^\varepsilon)=0,\qquad g'_2(\omega^\varepsilon)\geq 1-\varepsilon 
$$
for every $\varepsilon>0$.
\end{lemma}

\begin{proof}
Since $M$ is not Zoll, let us consider an arbitrary non-periodic ray $\gamma\in\Gamma$. Given any $k\in\N^*$, let us apply the construction of Lemma \ref{crucial_lemma} to the portion of $\gamma$ consisting of $\gamma([2^k,2^k+k])$: there exists a decreasing sequence $(\omega_k^\varepsilon)_{\varepsilon>0}$ of closed subsets such that
$$
\gamma([2^k,2^k+k])\subset M\setminus\omega_k^\varepsilon,\qquad
\vert M\setminus\omega_k^\varepsilon\vert\leq \frac{\varepsilon}{2^k}, \qquad
g_2^k(\omega_k^\varepsilon)=0,\qquad
g_2'(\omega_k^\varepsilon)\geq 1-\frac{\varepsilon}{2^k} .
$$
Setting $\omega^\varepsilon = \cap_{k\in\N^*}\omega_k^\varepsilon$, we have \eqref{strictinclusions} and $\vert M\setminus\omega^\varepsilon\vert \leq \sum_{k\in\N^*} \vert M\setminus\omega_k^\varepsilon\vert\leq \varepsilon$.
As in the proof of Lemma \ref{lem9zoll}, we have $g_2^k(\omega^\varepsilon)=0$ for every $k$ and thus $g_2(\omega^\varepsilon)=0$. 

It remains to prove that $g'_2(\omega^\varepsilon)\geq 1-\varepsilon$. 
Since $g_2'(\omega_k^\varepsilon)\geq 1-\frac{\varepsilon}{2^k}$, we have
$$
\forall\gamma'\in\Gamma\qquad \liminf_{T\rightarrow+\infty} \frac{1}{T} \int_0^T \chi_{\omega_k^\varepsilon}(\gamma'(t))\, dt \geq 1-\frac{\varepsilon}{2^k}
$$
and thus
$$
\forall\gamma'\in\Gamma\qquad \limsup_{T\rightarrow+\infty} \frac{1}{T} \int_0^T \chi_{M\setminus\omega_k^\varepsilon}(\gamma'(t))\, dt \leq \frac{\varepsilon}{2^k} .
$$
Since $M\setminus\omega^\varepsilon = \cup_{k\in\N^*} (M\setminus\omega_k^\varepsilon)$, we have $\chi_{M\setminus\omega^\varepsilon} \leq \sum_{k\in\N^*} \chi_{M\setminus\omega_k^\varepsilon}$ and thus
\begin{multline*}
\forall\gamma'\in\Gamma\qquad 
\limsup_{T\rightarrow+\infty} \frac{1}{T}\int_0^T \chi_{M\setminus\omega^\varepsilon}(\gamma'(t))\, dt
\leq \limsup_{T\rightarrow+\infty} \sum_{k\in\N^*} \frac{1}{T}\int_0^T \chi_{M\setminus\omega_k^\varepsilon}(\gamma'(t))\, dt \\
\leq  \sum_{k\in\N^*} \limsup_{T\rightarrow+\infty} \frac{1}{T}\int_0^T \chi_{M\setminus\omega_k^\varepsilon}(\gamma'(t))\, dt
\leq \sum_{k\in\N^*}\frac{\varepsilon}{2^k} = \varepsilon
\end{multline*}
and therefore
$$
\forall\gamma'\in\Gamma\qquad \liminf_{T\rightarrow+\infty} \frac{1}{T}\int_0^T \chi_{\omega^\varepsilon}(\gamma'(t))\, dt \geq 1-\varepsilon
$$
which gives $g'_2(\omega^\varepsilon)\geq 1-\varepsilon$. 
\end{proof}

\begin{remark}
Note that, in the above construction, we have $\gamma([2^k,2^k+k]) = \cap_{\varepsilon>0} (M\setminus\omega_k^\varepsilon)$ and thus
$$
\bigcup_{k\in\N^*} \gamma([2^k,2^k+k]) = M\setminus \bigcap_{k\in\N^*} \bigcup_{\varepsilon>0} \omega_k^\varepsilon .
$$
\end{remark}

\begin{lemma}\label{lemZollclosedsubsets}
$M$ is Zoll if and only if $g_2(\omega)=g_2'(\omega)$ for every closed subset $\omega\subset M$.
\end{lemma}

\begin{proof}
If $M$ is Zoll then, as a particular case of Lemma \ref{lemZollg2}, we obtain that $g_2(\omega)=g_2'(\omega)$ for every closed subset $\omega\subset M$.
If $M$ is not Zoll then by Lemma \ref{crucial_lemma_extended} there exists $\omega\subset M$ closed such that $g_2(\omega)=0$ and $g_2'(\omega)\geq 1/2$. The lemma is proved.
\end{proof}

\begin{lemma}
If $M$ is Zoll and is a two-point homogeneous space and if the Dirac $\delta_\gamma$ along any periodic geodesic is a QL then $\QL(S^*M)=\I(S^*M)$.
\end{lemma}

\begin{proof}
We follow \cite{Macia_CPDE2008} in which this result is proved and we only provide here a sketch of proof.

It suffices to prove that, given $N$ periodic geodesics $\tilde\gamma_1,\ldots,\tilde\gamma_N$, having the common period $T$, the convex combination $\sum_{i=1}^N c_i \delta_{\tilde\gamma_i}$, with $c_i>0$ and $\sum_{i=1}^N c_i=1$, is a QL. This is enough because, reasoning as in \cite{JakobsonZelditch}, by the Krein-Milman theorem, such convex combinations are dense in the set of invariant measures, which is a Hausdorff space (note that the Dirac along any geodesic is ergodic).

Given any two periodic geodesics $\tilde\gamma_1$ and $\tilde\gamma_2$, by two-point homogeneity there exists an isometry $\chi$ of $M$ mapping $\tilde\gamma_1$ to $\tilde\gamma_2$. This isometry lets the Laplacian invariant, and hence if $\phi$ is an eigenfunction then $\phi\circ\chi$ is as well an eigenfunction. Therefore, if $\phi^2$ is concentrated along $\tilde\gamma_1$ then $\phi^2\circ\chi$ is concentrated along $\tilde\gamma_2$. The conclusion is then easy by using that $\langle\Op(a)\phi_{j_k},\phi_{j_k}\circ\chi\rangle\rightarrow 0$. Following \cite[Proposition 3.3]{Gerard_SEDP}, this convergence is obtained by decomposing $a=a_1+a_2$ with $\delta_{\tilde\gamma_1}(a_1)\leq\varepsilon$ and $\delta_{\tilde\gamma_2}(a_2)\leq\varepsilon$, and by writing that $\vert\langle\Op(a)\phi_{j_k},\phi_{j_k}\circ\chi\rangle\vert \leq \vert\langle\Op(a_1)\phi_{j_k},\phi_{j_k}\circ\chi\rangle^2 + \vert\langle\phi_{j_k},\Op(a_2)\phi_{j_k}\circ\chi\rangle\vert^2 + \mathrm{o}(1)$ as $k\rightarrow+\infty$.
\end{proof}

We set
$$
\bar A_T(\omega) = \frac{1}{T}\int_0^Te^{-it\sqrt{\triangle}} \chi_\omega e^{it\sqrt{\triangle}}\, dt .
$$
By the Egorov theorem, we have $\sigma_P(\bar A_T(\omega)) = \frac{1}{T}\int_0^T \chi_\omega\circ\phi_t\, dt$ and $g_2^T(\omega) = \inf \frac{1}{T}\int_0^T \chi_\omega\circ\phi_t\, dt$.

\begin{lemma}
For every $y=\sum_{\lambda} P_\lambda y = \sum_{\lambda} y_\lambda\phi_\lambda\in L^2(M)$, we have
$$
\bar A_T(\omega) y = \frac{1}{T}\int_0^Te^{-it\sqrt{\triangle}} \chi_\omega e^{it\sqrt{\triangle}}\, dt\ y \underset{T\rightarrow+\infty}{\longrightarrow} \bar A_\infty(\omega) y = \sum_{\lambda} \left( y_\lambda \int_\omega\phi_\lambda^2\right) \phi_\lambda .
$$
\end{lemma}

In other words, the operator $\bar A_T$ converges pointwise to a diagonal operator in $L^2(M)$. Here, the $\phi_\lambda$ are eigenfunctions of norm $1$, and the sum runs over  \emph{distinct} eigenvalues; $P_\lambda$ is the projection onto the eigenspace corresponding to the eigenvalue $\lambda$.

\begin{proof}
We have
$$
\bar A_T(\omega) y = \sum_{\lambda} \langle \bar A_T(\omega) y, \phi_\lambda\rangle\phi_\lambda
= \sum_{\lambda} \left( \sum_{\mu} \frac{1}{T} \int_0^T e^{it(\mu-\lambda)}\, dt\ y_\mu \int_\omega \phi_\lambda\phi_\mu \right) \phi_\lambda
$$
Let us fix an integer $N>\lambda$. Setting $r_N =  \sum_{\mu>N} \frac{1}{T} \int_0^T e^{it(\mu-\lambda)}\, dt\ y_\mu \int_\omega \phi_\lambda\phi_\mu \ \in\C$, we have
$$
\langle \bar A_T(\omega) y, \phi_\lambda\rangle = \sum_{\mu\leq N} \frac{1}{T} \int_0^T e^{it(\mu-\lambda)}\, dt\ y_\mu \int_\omega \phi_\lambda\phi_\mu + r_N .
$$
Clearly, if $\lambda\neq\mu$ then $\frac{1}{T} \int_0^T e^{it(\mu-\lambda)}\, dt\rightarrow 0$ as $T\rightarrow +\infty$, and if $\lambda=\mu$ then $\frac{1}{T} \int_0^T e^{it(\mu-\lambda)}\, dt=1$. Therefore the limit of the finite sum above is equal to $y_\lambda \int_\omega\phi_\lambda^2$. Let us prove that $r_N$ is arbitrarily small if $N$ is large enough.\footnote{It is interesting to note that it is not obvious to prove that the series defining $r_N$ is convergent. Saying that $\vert\int_\omega\phi_\lambda\phi_\mu\vert\leq 1$ and that $\vert\frac{1}{T} \int_0^T e^{it(\mu-\lambda)}\, dt\vert\leq 1$ is not enough.}
Setting $y^N = \sum_{\mu>N} y_\mu \phi_\mu$ (highfrequency truncature), we have
$$
r_N = \frac{1}{T} \int_0^T \int_\omega \sum_{\mu>N} e^{it\mu} y_\mu \phi_\mu(x) e^{-it\lambda} \phi_\lambda(x)\, dx_g\, dt 
=  \frac{1}{T} \int_0^T \int_\omega (e^{it\sqrt{\triangle}}y^N)(x) e^{-it\lambda} \phi_\lambda(x)\, dx_g\, dt 
$$
and thus
$$
\vert r_N\vert \leq \frac{1}{T} \int_0^T \int_M \vert (e^{it\sqrt{\triangle}}y^N)(x)\vert \vert \phi_\lambda(x)\vert\, dx_g\, dt 
\leq \left( \frac{1}{T} \int_0^T \Vert e^{it\sqrt{\triangle}}y^N\Vert^2 \, dt \right)^{1/2} = \Vert y^N\Vert
$$
because $e^{it\sqrt{\triangle}}$ is an isometry in $L^2(M)$. 
Therefore $r_N$ is small if $N$ is large enough.

At this step, we have proved that $\langle \bar A_T(\omega) y, \phi_\lambda\rangle \rightarrow y_\lambda \int_\omega\phi_\lambda^2$ as $T\rightarrow+\infty$, and thus that $\bar A_T(\omega) y \rightharpoonup \bar A_\infty(\omega) y$ for the weak topology of $L^2(M)$.

Let us now split $y=y_N+y^N$, with $y_N = \sum_{\lambda\leq N} y_\lambda \phi_\lambda$ and $y^N = \sum_{\lambda>N} y_\lambda \phi_\lambda$. By compactness for frequencies lower than or equal to $N$, we have $\bar A_T(\omega) y_N \rightarrow \bar A_\infty(\omega) y_N$ for the strong topology of $L^2(M)$. Besides, noting that $\Vert \bar A_T(\omega)\Vert\leq 1$, we have $\Vert \bar A_T(\omega)y^N\Vert\leq \Vert y^N\Vert$, and since $\Vert y^N\Vert$ can be made arbitrarily small by taking $N$ large, the result follows.
\end{proof}

\begin{remark}
Note that 
$$
g_1(\omega) = \inf_{\lambda}\int_\omega\phi_\lambda^2 = \inf_{\Vert y\Vert=1} \langle \bar A_\infty(\omega)y,y\rangle .
$$
\end{remark}

\begin{remark}
Note that, setting $A=\Op(a)$, $\hat f_T(t) = \frac{1}{T}\chi_{[0,T]}(t)$ (i.e., $f_T(t)=\frac{1}{2\pi}e^{iTt/2}\mathrm{sinc}(Tt/2)$), using the spectral decomposition \eqref{specdecomp}, we have
$$
\bar A_T(a) = \frac{1}{T} \int_0^T e^{-it\sqrt{\triangle}} A e^{it\sqrt{\triangle}}\, dt = \int_\R \hat f_T(t) e^{-it\sqrt{\triangle}} A e^{it\sqrt{\triangle}}\, dt = \sum_{\lambda,\mu} f_T(\lambda-\mu) P_\lambda A P_\mu 
$$
with
$$
f_T(\lambda-\mu) = \frac{1}{2\pi T} \int_0^T e^{i(\lambda-\mu)\xi}\, d\xi = \frac{e^{iT(\lambda-\mu)}-1}{2i\pi T(\lambda-\mu)} ,
$$
meaning that $\langle \bar A_T(a) \phi_\mu,\phi_\lambda\rangle = f_T(\lambda-\mu) \langle A\phi_\mu,\phi_\lambda\rangle = f_T(\lambda-\mu) \int_M a\phi_\mu\phi_\lambda$ and that (since $P_\lambda A P_\mu y = y_\mu \langle A\phi_\mu,\phi_\lambda\rangle \phi_\lambda$)
$$
\bar A_T(a) y
= \sum_{\lambda,\mu} f_T(\lambda-\mu) y_\mu \langle A\phi_\mu,\phi_\lambda\rangle \phi_\lambda
= \sum_{\lambda} \left( \sum_\mu f_T(\lambda-\mu) y_\mu \int_M a\phi_\lambda\phi_\mu\right)\phi_\lambda .
$$
Besides, we have
$$
\bar A_\infty(a) = \sum_\lambda P_\lambda A P_\lambda,
$$
meaning that
$\bar A_\infty(a) y = \sum_\lambda \left( y_\lambda \int_M a\phi_\lambda^2\right)\phi_\lambda$.

Note that $f_T$ converges pointwise to $\chi_{\{0\}}$ as $T\rightarrow +\infty$ (we take $2\pi=1$...). This means that the coefficient $(\mu,\lambda)$ of the Gramian operator $\bar A_T(a)$, given by $\langle \bar A_T(a) \phi_\mu,\phi_\lambda\rangle = f_T(\lambda-\mu) \langle A\phi_\mu,\phi_\lambda\rangle = f_T(\lambda-\mu) \int_M a\phi_\mu\phi_\lambda$, converges to $\delta_{\lambda\mu}\langle A\phi_\lambda,\phi_\lambda\rangle = \delta_{\lambda\mu}\int_M a\phi_\lambda^2$ (with the Kronecker symbol). In other words, the Gramian operator $\bar A_T(a)$ converges weakly (i.e., coefficient by coefficient) to the diagonal operator $\bar A_\infty(a)$.

To understand the next lemma, we recall that, given a family of quadratic forms $q_T(y,z) = \langle A_Ty,z\rangle$, indexed by $T$, with $(A_T)_{j,k} = q_T(\phi_j,\phi_k)$ in a Hilbert basis, the strong convergence of $q_T$ to $q_\infty$ coincides with the strong convergence of $A_T$ to $A_\infty$. Indeed, $q_T$ continuous means that $\vert q_T(y,z)\vert \leq \Vert q_T\Vert\Vert y\Vert\Vert z\Vert$, and the strong convergence of $q_T$ to $q_\infty$ means that
$$
\lim_{T\rightarrow +\infty} \Vert q_T-q_\infty\Vert = \lim_{T\rightarrow +\infty}  \sup_{\Vert y\Vert=1, \Vert z\Vert=1} \vert q_T(y,z)-q_\infty(y,z)\vert = 0 
$$
and besides we have
$$
\Vert A_T-A_\infty\Vert = \sup_{\Vert y\Vert=1} \Vert A_Ty-A_\infty y\Vert  = \sup_{\Vert y\Vert=1, \Vert z\Vert=1} \vert \langle A_Ty-A_\infty y,z\rangle\vert = \Vert q_T-q_\infty\Vert.
$$
\end{remark}

\begin{lemma}
Under the spectral gap assumption, we have $\bar A_T(\omega) \underset{T\rightarrow+\infty}{\longrightarrow} \bar A_\infty(\omega)$ in operator norm, i.e., we have uniform convergence in $L^2(M)$.
\end{lemma}

\begin{proof}
It suffices to prove that
$$
\lim_{T\rightarrow +\infty} \sup_{\Vert y\Vert=1, \Vert z\Vert=1} \sum_{\lambda\neq\mu} f_T(\lambda-\mu) \langle A\phi_\lambda,\phi_\mu\rangle y_\lambda z_\mu  = 0.
$$
Since $\vert f_T(\lambda-\mu)\vert\leq \frac{2}{T\vert\lambda-\mu\vert}$, we have
$$
\left\vert \sum_{\lambda\neq\mu} f_T(\lambda-\mu) \langle A\phi_\lambda,\phi_\mu\rangle y_\lambda z_\mu \right\vert
\leq \frac{2}{T} \sum_{\lambda\neq\mu} \frac{\vert y_\lambda\vert \vert z_\mu\vert}{\vert\lambda-\mu\vert} \leq \frac{\mathrm{C}}{T} ,
$$
as a consequence of Montgomery and Vaughan's inequality (see \cite{MontgomeryVaughan} and see the next remark), for all $y,z\in L^2(M)$ of norm $1$, and with $C>0$ independent on $y$ and $z$. The result follows.
\end{proof}

\begin{remark}
We recall here that the well known Hilbert inequality states that
$$
\left\vert \sum_{j\neq k} \frac{a_jÊ\bar b_k}{j-k} \right\vert^2  \leq \pi^2 \sum_{j=1}^{+\infty} \vert a_j\vert^2 \sum_{j=1}^{+\infty} \vert b_j\vert^2 \qquad \forall (a_j)_{j\in\N},(b_j)_{j\in\N}\in\ell^2(\C).
$$
The same statement holds true with $j-k$ replaced with $j+k$.
The generalization proved in \cite{MontgomeryVaughan} by Montgomery and Vaughan states that, given $\lambda_1 < \cdots < \lambda_j < \cdots$ with $\lambda_{j+1}-\lambda_j\geq \delta >0$ for every $j$ (uniform gap), one has
$$
\left\vert \sum_{j\neq k} \frac{a_jÊ\bar b_k}{\lambda_j-\lambda_k} \right\vert^2  \leq \frac{\pi^2}{\delta^2} \sum_{j=1}^{+\infty} \vert a_j\vert^2 \sum_{j=1}^{+\infty} \vert b_j\vert^2 \qquad \forall (a_j)_{j\in\N},(b_j)_{j\in\N}\in\ell^2(\C).
$$
In passing, note that the following slight generalization is immediate: 

Let $(\lambda_j)_{j\in\N^*}$ be a sequence of real numbers for which there exists $\ell>0$ and $m\in\N^*$ such that the intersection of $\{ \lambda_j\ \mid\ j\in\N^*\}$ with any interval of length $\ell$ has at most $m$ elements (note that this property differs from uniform local finiteness). Setting $f(s)=\min(1/s,\beta)$ for some arbitrary $\beta>0$, there exists $C>0$ such that
$$
\left\vert \sum_{j,k}  f(\lambda_j-\lambda_k) a_j \bar a_k \right\vert^2 \leq C \sum_{j=1}^{+\infty} \vert a_j\vert^2 \sum_{j=1}^{+\infty} \vert b_j\vert^2 \qquad \forall (a_j)_{j\in\N},(b_j)_{j\in\N}\in\ell^2(\C).
$$
\end{remark}

%

Before giving the next result, we first recall a well known lemma on coherent states.

\begin{lemma}\label{lem_coherent_state}
Let $x_0\in\R^n$, $\xi_0\in\R^n$, and $k\in\N^*$. We define the coherent state
$$
u_k(x) = \left(\frac{k}{\pi}\right)^\frac{n}{4} e^{ik(x-x_0).\xi_0-\frac{k}{2}\Vert x-x_0\Vert^2}.
$$
Then $\Vert u_k\Vert=1$, and for every symbol $a$ on $\R^n$ of order $0$, we have
$$
\mu_k(a) = \langle\Op(a)u_k,u_k\rangle = a(x_0,\xi_0) + \mathrm{o}(1)
$$
as $k\rightarrow +\infty$. In other words, $\mu_k\rightharpoonup \delta_{(x_0,\xi_0)}$.
\end{lemma}

\begin{proof}[Proof of Lemma \ref{lem_coherent_state}.]
This lemma is well known and can be found for instance in \cite[Chapter 5, Example 1]{Zworski}. We include a proof for convenience.

First of all, we compute\footnote{Here, we use the fact that $\int_{\R^n} e^{-\alpha\Vert x\Vert^2}\, dx = \left(\frac{\pi}{\alpha}\right)^\frac{n}{2}$.}
$
\Vert u_k\Vert^2 = \left(\frac{k}{\pi}\right)^\frac{n}{2} \int e^{-\frac{k}{2}\Vert x-x_0\Vert^2}\, dx = 1.
$
Now, by definition, we have
\begin{equation*}
\begin{split}
\langle\Op(a)u_k,u_k\rangle
&= \int \Op(a)u_k(x) \overline{u_k(x)}\, dx \\
&= \frac{1}{(2\pi)^n} \iiint e^{i(x-y).\xi} a(x,\xi) u_k(y) \overline{u_k(x)}\, dx\, dy\, d\xi \\
&= \frac{k^n}{(2\pi)^n} \iiint e^{ik(x-y).\xi} a(x,\xi) u_k(y) \overline{u_k(x)}\, dx\, dy\, d\xi
\end{split}
\end{equation*}
by the change of variable $\xi\mapsto k\xi$, and using the homogeneity of $a$. Then we get
\begin{equation*}
\begin{split}
\langle\Op(a)u_k,u_k\rangle
&= \frac{k^\frac{3n}{2}}{2^n\pi^\frac{3n}{2}} \iiint a(x,\xi) e^{ik(x-y).\xi} e^{ik(y-x).\xi_0} e^{-\frac{k}{2} (  \Vert x-x_0\Vert^2 + \Vert y-x_0\Vert^2 )} \, dx\, dy\, d\xi \\
&= \frac{k^\frac{3n}{2}}{2^n\pi^\frac{3n}{2}} \iint a(x,\xi) e^{-\frac{k}{2} \Vert x-x_0\Vert^2} \int e^{ik(x-y).\xi} e^{ik(y-x).\xi_0} e^{-\frac{k}{2} \Vert y-x_0\Vert^2 } \, dy \, dx\, d\xi .
\end{split}
\end{equation*}
Now, we have\footnote{Note that $\mathcal{F}(e^{-\alpha\Vert x\Vert^2})(\xi) = \left(\frac{\pi}{\alpha}\right)^\frac{n}{2}e^{-\frac{\Vert\xi\Vert^2}{4\alpha}}$.}
\begin{equation*}
\begin{split}
\int e^{ik(x-y).\xi} e^{ik(y-x).\xi_0} e^{-\frac{k}{2} \Vert y-x_0\Vert^2 } \, dy
&= e^{ik(x-x_0).(\xi-\xi_0)} \int e^{-ik(y-x_0).(\xi-\xi_0)} e^{-\frac{k}{2} \Vert y-x_0\Vert^2 } \, dy \\
&= e^{ik(x-x_0).(\xi-\xi_0)} \int e^{-iky.(\xi-\xi_0)} e^{-\frac{k}{2} \Vert y\Vert^2 } \, dy \\
&= e^{ik(x-x_0).(\xi-\xi_0)} \mathcal{F}(e^{-\frac{k}{2}\Vert y\Vert^2})(k(\xi-\xi_0)) \\
&= \left( \frac{2\pi}{k}\right)^\frac{n}{2} e^{ik(x-x_0).(\xi-\xi_0)} e^{-\frac{k}{2}\Vert \xi-\xi_0\Vert^2}
\end{split}
\end{equation*}
and therefore,
\begin{equation*}
\begin{split}
\langle\Op(a)u_k,u_k\rangle
&= \frac{k^n}{2^\frac{n}{2}\pi^n} \iint a(x,\xi) e^{ik(x-x_0).(\xi-\xi_0)} e^{-\frac{k}{2} (\Vert x-x_0\Vert^2 + \Vert \xi-\xi_0\Vert^2 )}\, dx\, d\xi  \\
&= \frac{k^n}{2^\frac{n}{2}\pi^n} a(x_0,\xi_0) \iint e^{ik(x-x_0).(\xi-\xi_0)} e^{-\frac{k}{2} (\Vert x-x_0\Vert^2 + \Vert \xi-\xi_0\Vert^2 )}\, dx\, d\xi  +  \mathrm{o}(1) \\
&= c_n a(x_0,\xi_0) +  \mathrm{o}(1)
\end{split}
\end{equation*}
as $k\rightarrow +\infty$. Moreover, taking $a=1$ in the above reasoning, we see that
$$
c_n = \iint e^{ikx.\xi} e^{-\frac{k}{2} (\Vert x\Vert^2 + \Vert \xi\Vert^2 )}\, dx\, d\xi = 1.
$$
The lemma is proved.
\end{proof}

\begin{remark}\label{rem_GB_highfrequency}
Let us construct $y^N\in L^2(M)$ as an approximation of $u_k$ having only frequencies larger than $N$. 
We consider the above solution $u_k$, defined on $M$ in a local chart around $(x_0,\xi_0)$ (we multiply the above expression by a function of compact support taking the value $1$ near $(x_0,\xi_0)$, and we adapt slightly the constant so that we still have $\Vert u_k\Vert=1$). Note that $\int_M u_k = \frac{2^\frac{n}{2}\pi^\frac{n}{4}}{k^\frac{n}{4}}$.
Now, we set
$$
\pi_N u_k = \sum_{j=1}^N \langle u_k,\phi_j\rangle \phi_j = \sum_{j=1}^N \int_M u_k(x)\phi_j(x)\, dx_g\ \phi_j .
$$
By usual Sobolev estimates and by the Weyl law, there exists $C>0$ such that $\Vert\phi_j\Vert_{L^\infty(M)}\leq C \lambda_j^{\frac{n}{2}}$ and $\lambda_j\sim j^{\frac{2}{n}}$ for every $j\in\N^*$, hence $\Vert\phi_j\Vert_{L^\infty(M)}\leq C j$. We infer that $\vert\langle u_k,\phi_j\rangle\vert\leq CN \int_M \vert u_k\vert\leq C2^\frac{n}{2}\pi^\frac{n}{4}\frac{N}{k^\frac{n}{4}}$ for every $j\leq N$.

Let $\varepsilon>0$ be arbitrary. Choosing $k$ large enough so that $C 2^\frac{n}{2}\pi^\frac{n}{4}\frac{N^2}{k^\frac{n}{4}}\leq \varepsilon$, we have $\Vert \pi_N u_k\Vert\leq\varepsilon$.
Setting $y^N = u_k - \pi_N u_k$, we have 
\begin{equation*}
\langle \Op(a) y^N, y^N\rangle
= \underbrace{\langle \Op(a) u_k, u_k\rangle}_{\simeq g_2^T(a)} + \underbrace{\langle \Op(a) \pi_N u_k, \pi_N u_k\rangle}_{\leq \varepsilon^2\max\bar a_T} 
- \underbrace{\langle \Op(a) \pi_N u_k, u_k\rangle}_{\vert\cdot\vert\leq \varepsilon\max a} - \underbrace{\langle \Op(a) u_k, \pi_N u_k\rangle}_{\vert\cdot\vert\leq \varepsilon\max a} 
\end{equation*}
and the conclusion follows.
\end{remark}

\begin{remark}\label{remark_GB}
In view of the next result, what is interesting to note that $e^{it\sqrt{\triangle}}u_k$ (or, accordingly, $e^{it\sqrt{\triangle}}(u_k-\pi_Nu_k)$) is a half-wave Gaussian beam along the geodesic $\varphi_t(x_0,\xi_0)$.

Indeed, for any symbol of order $0$, recalling that $A_t = e^{-it\sqrt{\triangle}}\Op(a)e^{it\sqrt{\triangle}}$ has $a_t = a\circ\varphi_t$ as principal symbol, we have $\langle\Op(a)e^{it\sqrt{\triangle}}u_k,e^{it\sqrt{\triangle}}u_k\rangle = \langle A_t u_k,u_k\rangle = \langle \Op(a_t) u_k,u_k\rangle+\mathrm{o}(1) = a_t(x_0,\xi_0)+\mathrm{o}(1)$ (by Lemma \ref{lem_coherent_state}), which means that $e^{it\sqrt{\triangle}}u_k$ is microlocally concentrated around $\varphi_t(x_0,\xi_0)$.
\end{remark}

\begin{lemma}
If the spectrum is \textit{uniformly locally finite}, then for every ray $\gamma$ (periodic or not), for every $T>0$, there exists $\nu\in\QL(M)$ such that $\nu(\gamma_{[0,T]})>0$.
Moreover, as a consequence, $M$ is Zoll.
\end{lemma}


\begin{proof}
By contradiction, let $\gamma$ be a ray and let $T>0$ such that $\nu(\gamma([0,T]))=0$ for every $\nu\in\QL(M)$. 
The proof goes in several steps.

\medskip

For every $\varepsilon>0$, we set $\omega_\varepsilon = \{x\in M\ \mid\ d(x,\gamma([0,T])<\varepsilon\}$ (open $\varepsilon$-neighborhood of $\gamma([0,T])$). We have $\gamma([0,T]) = \cap_{\varepsilon>0} (\omega_\varepsilon) = \cap_{\varepsilon>0} (\overline{\omega}_\varepsilon)$.
Note that $\nu(\omega_\varepsilon) \rightarrow \nu(\gamma([0,T]))=0$ as $\varepsilon\rightarrow 0$, for any $\nu\in\QL(M)$.

\medskip


Given any eigenvalue $\lambda\in\mathrm{Spec}(\sqrt{\triangle})$ and any associated eigenfunction $\phi_\lambda$, we set $\nu_\lambda=\phi_\lambda^2\, dx_g$.
For $\varepsilon>0$ fixed, we have $\limsup_{\lambda\rightarrow+\infty}\nu_\lambda(\omega_\varepsilon) \leq \sup_{\nu\in\QL(M)} \nu(\omega_\varepsilon)$ (the $\sup$ is a $\max$ since $\QL(S^*M)$ is compact) which converges to $0$ as $\varepsilon\rightarrow 0$. It follows that the double limit below exists and
$$
\lim_{\stackrel{\varepsilon\rightarrow 0}{\lambda\rightarrow+\infty}}\int_{\omega_\varepsilon}  \phi_\lambda^2\, dx_g
= \lim_{\varepsilon\rightarrow 0}\lim_{\lambda\rightarrow+\infty}\int_{\omega_\varepsilon}  \phi_\lambda^2\, dx_g
= \lim_{\lambda\rightarrow+\infty}\lim_{\varepsilon\rightarrow 0} \int_{\omega_\varepsilon}  \phi_\lambda^2\, dx_g = 0 .
$$
In passing, note that, for any $\lambda$ fixed, we have of course $\lim_{\varepsilon\rightarrow 0}\int_{\omega_\varepsilon} \phi_\lambda^2\, dx_g = 0$.

Now, since $\left| \int_{\omega_\varepsilon} \phi_\lambda \phi_\mu \, dx_g \right| \leq \frac{1}{2} \left( \int_{\omega_\varepsilon} \phi_\lambda^2 \, dx_g + \int_{\omega_\varepsilon} \phi_\mu^2 \, dx_g \right) $,
we infer that
\begin{equation}\label{limepsij}
\lim_{\stackrel{\varepsilon\rightarrow 0}{\lambda,\mu\rightarrow+\infty}}\int_{\omega_\varepsilon} \phi_\lambda\phi_\mu\, dx_g = 0.
\end{equation}
In passing, note that, in the above limit, we can even keep one of the indices $\lambda$ or $\mu$ fixed.

\medskip


We set $E(N,\varepsilon) = \int_\R \int_{\omega_\varepsilon}  \hat{f}(t) |e^{it\sqrt{\triangle}}y^N|^2 \,dx_g\, dt$ for some $y^N\in L^2(M)$ of norm $1$, and $f$ is an arbitrary smooth function of compact support, such that $f(0)=1$ (and thus, $\int_\R\hat{f}(t)\, dt=1$, up to constant), such that $\hat{f}\geq 0$, and such that $\hat{f}(0)=\int_\R f(t)\, dt >0$.
Note that $E(N,\varepsilon) = 1-E^c(N,\varepsilon) $ with $E^c(N,\varepsilon) = \int_\R \int_{M\setminus\omega_\varepsilon}  \hat{f}(t) |(e^{it\sqrt{\triangle}}y^N)(t,x)|^2 \,dx_g\, dt$.

We choose the high-frequency coherent state $y^N$ constructed in Remark \ref{rem_GB_highfrequency}, so that (by Remark \ref{remark_GB}) $e^{it\sqrt{\triangle}}y^N$ is a half-wave kind of Gaussian beam along $\gamma(t)$, microlocally concentrated at $\varphi_t(x_0,\xi_0)$, where $(x_0,\xi_0)\in S^*M$ is the initial condition corresponding to the ray $\gamma$.

Expanding in series, we have $y^N = \sum_{\lambda\geq N} P_\lambda y^N$, and we set $P_\lambda y^N=a_\lambda \phi_\lambda$ for some eigenfunction $\phi_\lambda$ of norm $1$ associated with $\lambda$, with $\sum\vert a_\lambda\vert^2=1$. Then, since $e^{it\sqrt{\triangle}}y^N=\sum_{\lambda\geq N} a_\lambda e^{it\lambda}\phi_\lambda$, we get
$$
E(N,\varepsilon) = \sum_{\stackrel{\lambda,\mu\in\mathrm{Spec}(\sqrt{\triangle})}{\lambda,\mu\geq N}} f(\lambda-\mu) a_\lambda \bar a_\mu \int_{\omega_\varepsilon} \phi_\lambda\phi_\mu \, dx_g .
$$
Note that, in the sum, by definition of $P_\lambda$, each diagonal term ($\lambda=\mu$) is single.
At this step, in order to use \eqref{limepsij}, we need to assume that
$$
\sup_{\underset{\lambda\in\mathrm{Spec}(\sqrt{\triangle})}{\sum}\vert a_\lambda\vert^2=1} \ \sum_{\stackrel{\lambda,\mu\in\mathrm{Spec}(\sqrt{\triangle})}{\lambda,\mu\geq N}} \vert f(\lambda-\mu)\vert \vert a_\lambda\vert \vert a_\mu\vert < +\infty.
$$
In particular, this is true under the spectral assumption of uniform local finiteness of the spectrum. Indeed, since $f$ has a compact support, there exists $m\in\N^*$ such that the sum is bounded above by
$$
\sum_j\left( \vert a_j\vert \vert a_{j-m}\vert +\cdots+ \vert a_j\vert^2+\cdots+ \vert a_j\vert \vert a_{j+m}\vert \right) \leq (2m+1)\sum_j\vert a_j\vert^2=(2m+1).
$$
Then, using \eqref{limepsij} we infer that
\begin{equation}\label{limE1}
\lim_{\stackrel{\varepsilon\rightarrow 0}{N\rightarrow+\infty}} E^c(N,\varepsilon) = \lim_{\varepsilon\rightarrow 0}\lim_{N\rightarrow+\infty} E^c(N,\varepsilon) = \lim_{N\rightarrow+\infty}\lim_{\varepsilon\rightarrow 0} E^c(N,\varepsilon) = 1.
\end{equation}
%
%
Now, let $T>0$ be arbitrary. 
Since we have taken $f$ such that $\hat{f}\geq 0$, $\hat{f}(0)>0$ and $\int_\R \hat{f}(t) \, dt=1$, we have $\int_{\R\setminus[0,T]}\hat{f}(t)\, dt <1$.
Therefore, writing
$$
E^c(N,\varepsilon) = \int_{\R\setminus[0,T]} \int_{M\setminus\omega_\varepsilon}  \hat{f}(t) |e^{it\sqrt{\triangle}}y^N|^2 \,dx_g\, dt + \int_{[0,T]} \int_{M\setminus\omega_\varepsilon}  \hat{f}(t) |e^{it\sqrt{\triangle}}y^N|^2 \,dx_g\, dt ,
$$
the first integral is (strictly) less than $1$, and the second integral converges to $0$ as $N\rightarrow +\infty$ and $\varepsilon>0$ is chosen small enough, because $e^{it\sqrt{\triangle}}y^N$ concentrates along $\gamma([0,T])$ that does not meet $M\setminus\omega_\varepsilon$ by construction.
We have obtained a contradiction with \eqref{limE1}.

An example of function $f$ that is adequate is a triangle function, but there are many possibilities. Also, note that ensuring nonnegativity is never a problem, by taking the convolution by itself of a given function.

The last part is now an obvious consequence. Indeed, if there were to exist a non-periodic geodesic, then there exists a QL having positive mass along $\gamma$. By invariance, this mass must be equal to $+\infty$, which is absurd. Therefore all geodesics are periodic, and hence, as already mentioned (Wadsley's Theorem, see \cite{Besse}), $M$ is Zoll.
\end{proof}

\begin{lemma}
If $M$ has the spectral gap property, then $M$ is Zoll and the Dirac along any periodic ray is a QL.
\end{lemma}

\begin{proof}
To prove that the spectral gap property implies that $M$ is Zoll, one can apply the theory done  in \cite{Guillemin_Duke,Helton}: under the spectral gap property, $\Sigma\neq\R$, and hence $M$ is Zoll. Or, one can say that spectral gap implies uniformly locally finite spectrum, and then $M$ is Zoll by the previous item. 

The fact that the Dirac $\delta_\gamma$ along any periodic ray $\gamma$ is a QL is proved in \cite{Macia_CPDE2008}. We provide however a proof of that fact hereafter.

Let $\gamma$ be a periodic ray. As in the previous proof, let $y^N\in L^2(M)$ of norm $1$ such that $e^{it\sqrt{\triangle}} y^N\rightarrow \gamma(t)$ uniformly along $[0,T]$, as $N\rightarrow +\infty$.
Using the already defined operator $A_f = \int_\R \hat{f}(t) e^{-it\sqrt{\triangle}} A e^{it\sqrt{\triangle}} \,dt$ with $A\in\Psi^0$ of principal symbol $a$, we have
$$
\langle A_f y^N, y^N\rangle = \langle \Op(a_f) y^N, y^N\rangle = \int_\R \hat{f}(t) \left\langle \Op(a) e^{it\sqrt{\triangle}} y^N, e^{it\sqrt{\triangle}} y^N\right\rangle\, dt.
$$
Since $\sigma_P(A_f)=\int_\R \hat f(t)\, a\circ\varphi_t\, dt$, we have
$$
\lim_{N\rightarrow +\infty} \langle A_f y^N, y^N\rangle = \int_\R \hat{f}(t) \, a\circ\varphi_t(z)\, dt 
$$
where $z\in S^*M$ is the initial condition associated with $\gamma$.

Besides, since $A_f = \sum_{\lambda,\mu\in\mathrm{Spec}(\sqrt{\triangle})} f(\lambda-\mu) P_\lambda A P_\mu$, we have also
$$
\langle A_f y^N, y^N\rangle = \sum_{\stackrel{\lambda,\mu\in\mathrm{Spec}(\sqrt{\triangle})}{\lambda,\mu\geq N}} f(\lambda-\mu) a_\lambda^N \bar a_\mu^N \langle A\phi_\lambda,\phi_\mu\rangle
$$
where we have assumed that $y^N = \sum_{\stackrel{\lambda\in\mathrm{Spec}(\sqrt{\triangle})}{\lambda\geq N }}  a_\lambda^N \phi_\lambda$. 

Since $M$ has the spectral gap property, if we choose $f$ of compact support centered at $0$, such that $f(0)=1$, and the size of support is smaller than the gap, then the above sum reduces to
$$
\langle A_f y^N, y^N\rangle = \sum_{\stackrel{\lambda\in\mathrm{Spec}(\sqrt{\triangle})}{\lambda\geq N }}  \vert a_\lambda^N\vert^2 \langle A\phi_\lambda,\phi_\lambda\rangle.
$$
Now, in the above formulas, we take $a=a^\varepsilon$, with $a^\varepsilon$ supported outside of a tubular neighborhood of $\gamma$, and essentially equal to $1$. We set $A^\varepsilon=\Op(a^\varepsilon)$, and accordingly, we consider the operator $A_f^\varepsilon$. Then the above sum tends to zero as $N\rightarrow +\infty$. Since this is a (Cesaro-like) convex combination of nonnegative terms, reasoning by contradiction it follows that there exists a subsequence such that $\langle A^\varepsilon\phi_{j_k},\phi_{j_k}\rangle\rightarrow 0$ as $k\rightarrow +\infty$, for every $\varepsilon$. This gives the result.

To be more precise: what we have proved is that for every open subset $\omega_\varepsilon\subset M\setminus\gamma(\R)$, there exists a subsequence such that $\int_{\omega_\varepsilon}\phi_{j_k}^2\, dx_g\rightarrow 0$. Extracting a subsequence if necessary, this means that there exists a QL $\nu_\varepsilon$ such that $\nu_\varepsilon(\omega_\varepsilon)=0$. In other words, we have found a decreasing sequence of closed sets $\gamma([0,T])\subset F_k$ (converging to $\gamma([0,T])$) such that $\nu_k(F_k)=1$. Now, let $\nu$ be a weak limit of $\nu_k$. We have $\nu(\gamma([0,T])) = \lim_{k\rightarrow+\infty} \nu(F_k)$, but $\nu(F_k)\geq \limsup_{j\rightarrow +\infty} \nu_j(F_k)$, and for $j>k$ we have $F_j\subset F_k$ and thus $\nu_j(F_k)\geq \nu_j(F_j)=1$, whence finally $\nu(F_k)=1$ and thus $\nu(\gamma([0,T]))=1$. We conclude that $\nu$ is the Dirac along $\gamma$.
\end{proof}

\appendix

\section{Appendix}
\subsection{Portmanteau theorem}\label{app:Portmanteau}
Let us recall the Portmanteau theorem (see, e.g., \cite{Billingsley}).
Let $X$ be a topological space, endowed with its Borel $\sigma$-algebra. 
Let $\mu$ and $\mu_n$, $n\in\N^*$, be finite Borel measures on $X$. Then the following items are equivalent:
\begin{itemize}
\item $\mu_n\rightarrow\mu$ for the narrow topology, i.e., $\int f\, d\mu_n\rightarrow\int f\, d\mu$ for every bounded continuous function $f$ on $X$;
\item $\int f\, d\mu_n\rightarrow\int f\, d\mu$ for every Borel bounded function $f$ on $X$ such that $\mu(\Delta_f)=0$, where $\Delta_f$ is the set of points at which $f$ is not continuous;
\item $\mu_n(B)\rightarrow\mu(B)$ for every Borel subset $B$ of $X$ such that $\mu(\partial B)=0$;
\item $\mu(F)\geq\limsup\mu_n(F)$ for every closed subset $F$ of $X$, and $\mu_n(X)\rightarrow\mu(X)$;
\item $\mu(O)\leq\liminf\mu_n(O)$ for every open subset $O$ of $X$, and $\mu_n(X)\rightarrow\mu(X)$.
\end{itemize}

\subsection{Some facts on invariant measures}\label{app:measures}
%
We recall that, given a periodic ray $\gamma$ on $M$, the Dirac measure $\delta_\gamma$ on $M$ is defined by $\delta_\gamma(f)=\frac{1}{T}\int_0^Tf(\gamma(t))\, dt$, for every $f\in C^0(M)$, where $T$ is the period of $\gamma$.
If $\gamma=\pi\circ\tilde\gamma$ where $\tilde\gamma$ is the periodic geodesic in $S^*M$ projecting onto $\gamma$, then we consider as well the Dirac measure $\delta_{\tilde\gamma}$ on $S^*M$  defined by $\delta_{\tilde\gamma}(a)=\frac{1}{T}\int_0^Ta(\tilde\gamma(t))\, dt$, for every $a\in C^0(S^*M)$. Obviously, we have $\pi_*\delta_{\tilde\gamma}=\delta_\gamma$.

Before stating the next result, we recall a useful fact. Let $\Phi:X\rightarrow Y$ be a measurable mapping, with $X$ and $Y$ separable metric spaces. Let $\mu$ be a Radon measure on $X$ and let $\Phi_*\mu$ be its pushforward to $Y$ under $\Phi$. 
We recall that the support of $\mu$ is the closed subset $\supp(\mu)$ of $X$ defined as the set of all $x\in X$ such that $\mu(U)>0$ for any neighborhood $U$ of $x$. If $\Phi$ is continuous and proper then $\Phi(\supp(\mu))=\supp(\Phi_*\mu)$.

Hereafter, we establish a decomposition of invariant probability Radon measures with respect to any Dirac measure along a periodic ray.
Note that, by propagation, a finite invariant Radon measure can involve a Dirac part only if this part is a Dirac $\delta_{\tilde\gamma}$ along a geodesic and moreover the geodesic $\tilde\gamma$ has to be periodic (due to finiteness of the measure). Recall that $\I(S^*M)$ designates the set of invariant probability Radon measures on $S^*M$.

\begin{proposition} \label{propQLM}
Let $\mu \in \I(S^*M)$ and let $\gamma=\pi\circ\tilde\gamma\in\Gamma$ be a periodic ray. Then:
\begin{itemize}
\item $\mu=\delta_{\tilde\gamma}$ if and only if $\pi_*\mu = \delta_\gamma$. 
\item There exists a nonnegative Radon measure $\mu_1$ on $S^*M$ that is invariant under the geodesic flow and satisfies $\mu_1(\tilde\gamma(\R))=0$ and $(\pi_*\mu_1)(\gamma(\R)) =0$, such that $\mu=\mu_1 + a \delta_{\tilde\gamma}$, with $a=\mu(\tilde\gamma(\R))$.
 \end{itemize}
\end{proposition}

\begin{remark} \label{corQLM}
It follows from the second point of the above proposition that $\pi_*\mu = \pi_*\mu_1 + a\delta_\gamma$.
\end{remark}

\begin{proof}[Proof of Proposition \ref{propQLM}.]
Let $\mu \in \I(S^*M)$, and let $\gamma=\pi\circ\tilde\gamma\in\Gamma$ be a periodic ray, of period $T$.

Let us prove the first point.
If $\mu=\delta_{\tilde\gamma}$, then $\pi_*\mu=\pi_*\delta_{\tilde\gamma}=\delta_\gamma$. 
Conversely, if $\pi_*\mu = \delta_\gamma=\pi_*\delta_{\tilde\gamma}$, then $\supp(\mu)\subset\pi^{-1}(\gamma(\R))$. Let us prove that we have exactly $\supp(\mu)=\tilde\gamma(\R)$. 
Since $\mu$ is invariant under the geodesic flow $\varphi_t$ (meaning that $\varphi_t^*\mu=\mu$ for any $t\in\R$), we must have $\varphi_t(\supp(\mu))=\supp(\mu)\subset\pi^{-1}(\gamma(\R))$ for any $t\in\R$.
Given any $z\in\supp(\mu)$, we must have $\pi(z)\in\pi(\supp(\mu))=\supp(\pi_*\mu)=\gamma(\R)$, and therefore there exists $s_0\in\R$ such that $\pi(z)=\gamma(s_0)$. Since $\varphi_t(z)\in\pi^{-1}(\gamma(\R))$ for any $t\in\R$, we must have $\pi(\varphi_t(z))=\gamma(s(t))$ for some $s(t)$, with $s(0)=s_0$. The function $t\mapsto s(t)$ must be strictly monotone around $s_0$, and since the ray $\gamma$ cannot have two distinct extremal lifts, it follows that $z\in\tilde\gamma$ and $\varphi_t(z)=\tilde\gamma(s_0+t)$. The first point is proved.

In order to prove the second point, we are going to use a general result of Riemannian geometry.

\begin{lemma}\label{lem_geomR}
We denote by $X$ the Hamiltonian geodesic field. Given any geodesic ray $\gamma=\pi\circ\tilde\gamma\in\Gamma$, $X$ is transverse to $\pi^{-1}(\gamma(\R))\setminus\tilde\gamma(\R)$.
\end{lemma}

Note that $\pi^{-1}(\gamma(\R))$ is a fiber bundle with base the one-dimensional manifold $\gamma(\R)$.

\begin{proof}[Proof of Lemma \ref{lem_geomR}.]
Denoting by $g$ the Riemannian metric on $M$ and by $g^*$ the cometric, we have $H(x,\xi)=\frac{1}{2}g^*_x(\xi,\xi)$ in local symplectic coordinates $(x,\xi)\in S^*M$ and $X=(\partial_\xi H,-\partial_xH)$. We note that if $\xi_1\neq\xi_2$ then $\partial_\xi H(x,\xi_1)\neq \partial_\xi H(x,\xi_2)$. 
Taking a point $x=\pi(x,\xi_1)$ on the ray $\gamma$, the tangent space to $\pi^{-1}(\gamma(\R))$ at any point $(x,\xi)$ of the vertical fiber above $x$ does not depend on $\xi$ and is spanned by all $(\partial_\xi H(x,\xi_1),*)$. The previous remark shows that $X$ is tranverse to $\pi^{-1}(\gamma(\R))$ at $(x,\xi_2)$ (see Figure \ref{fig_geod}). 
\end{proof}

\begin{figure}[h]
\centerline{\scalebox{0.37}{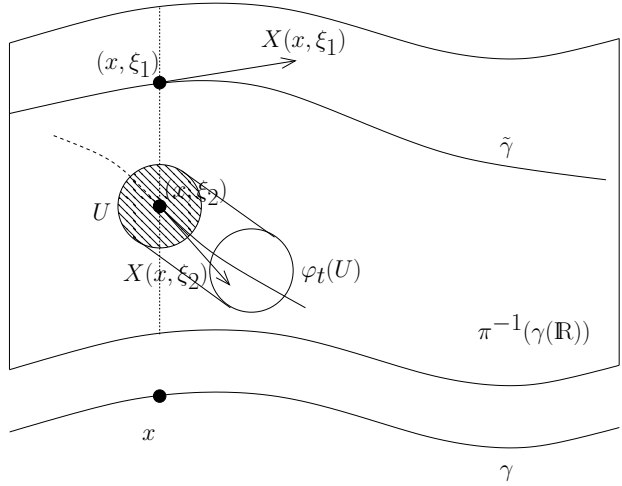}}
\caption{Illustration of Lemmas \ref{lem_geomR} and \ref{lemraynul}.}
\label{fig_geod}
\end{figure}

From Lemma \ref{lem_geomR}, we deduce another general lemma.

\begin{lemma}\label{lemraynul}
Let $\mu$ be a nonnegative finite invariant Radon measure on $S^*M$. Given any geodesic ray $\gamma=\pi\circ\tilde\gamma\in\Gamma$, we have $\mu(\pi^{-1}(\gamma(\R))\setminus\tilde\gamma(\R))=0$.
\end{lemma}

\begin{proof}[Proof of Lemma \ref{lemraynul}.]
We argue by contradiction. If $\mu(\pi^{-1}(\gamma(\R))\setminus\tilde\gamma(\R))>0$, then there exist $z\in\pi^{-1}(\gamma(\R))\setminus\tilde\gamma(\R)$ and a neighborhood $U$ of $z$ in the manifold $\pi^{-1}(\gamma(\R))$ such that $\mu(U)>0$. Let us propagate $U$ under the geodesic flow $\varphi_t$. By Lemma \ref{lem_geomR}, the Hamiltonian geodesic field $X$ is transverse to $\pi^{-1}(\gamma(\R))$. Hence if $U$ and $t>0$ are sufficiently small then $\varphi_t(U) \cap \pi^{-1}(\gamma(\R)) = \emptyset$, and actually the union of all $\varphi_s(U)$ with $0\leq s\leq t$ is a cylinder (denoted by $\mathcal{C}$) with distinct layers $\varphi_s(U)$ (see Figure \ref{fig_geod}). 

Now, since $\mu$ is invariant under the geodesic flow, we have $\mu(\varphi_s(U))=\mu(U)>0$. It follows that $\mu(\mathcal{C})=+\infty$, which contradicts the fact that $\mu$ is finite.
\end{proof}

We are now in a position to prove the second point. We set $a=\mu(\tilde\gamma(\R))$ and $\mu_1=\mu-a\delta_{\tilde\gamma}$. Given any measurable subset $B$ of $S^*M$, by definition of $\mu_1$ we have $\mu_1(B)=\mu_1(B\setminus\tilde\gamma(\R))+\mu_1(\tilde\gamma(\R))=\mu(B\setminus\tilde\gamma(\R))$, from which we infer that $\mu_1$ is a nonnegative measure that is invariant under the geodesic flow.
Note that $\mu_1(\tilde\gamma(\R))=0$.

Let us prove that $(\pi_*\mu_1)(\gamma(\R))=0$. This follows from Lemma \ref{lemraynul}, by writing that $\pi_*\mu_1=\pi_*\mu-a\delta_\gamma$ (because $\pi_*\delta_{\tilde\gamma} = \delta_\gamma$) and
$$
(\pi_*\mu_1)(\gamma(\R)) = (\pi_*\mu)(\gamma(\R)) - a
= \mu(\pi^{-1}(\gamma(\R))) - \mu(\tilde\gamma(\R))
= \mu(\pi^{-1}(\gamma(\R))\setminus\tilde\gamma(\R))=0.
$$
The proposition is proved.
\end{proof}

\end{document}